\documentclass[12pt]{article}
\usepackage{amssymb,amsmath,amsthm,color,ulem}

\newtheorem{formula}{}
\newtheorem{proposition}[formula]{Proposition}
\newtheorem{corollary}[formula]{Corollary}
\newtheorem{lemma}[formula]{Lemma}
\newtheorem{thm}[formula]{Theorem}
\newtheorem{example}{Example} 

\newcommand{\abs}[1]{{\lvert #1 \rvert}}
\renewcommand \b {\mathbf{B}}
\newcommand \ep{\varepsilon}
\renewcommand \l {\lambda}
\newcommand \bn{\mathbf{N}}
\newcommand \ab{$(A, \b)$}
\newcommand \ta{table algebra}

\newcommand \si{\sigma}
\newcommand \mbr{\mathbb{R}}

\DeclareMathOperator{\aut}{Aut}
\DeclareMathOperator{\irr}{Irr}
\DeclareMathOperator{\spa}{Span}

\renewcommand \c{\mathbb C}
\newcommand \cg{$(\c G, G)$}

\newcommand{\R}{\mathbb R}

\begin{document}
\title{Noncommutative reality-based algebras of rank $6$} 

\author{{Allen Herman\thanks{The first author was supported by an NSERC Discovery Grant.}}, {Mikhael Muzychuk\thanks{The second author was supported by the Wilson Endowment of Eastern Kentucky University.}}, and {Bangteng Xu}}

\date{Submitted: August 30, 2016; Revised March 9, 2017 }
\maketitle

\begin{abstract}
We show that noncommutative standard reality-based algebras (RBAs) of dimension $6$ are determined up to exact isomorphism by their character tables.   We show that the possible character tables of these RBAs are determined by seven real numbers, the first four of which are positive and the remaining three real numbers can be arbitrarily chosen up to a single exception.  We show how to obtain a concrete matrix realization of the elements of the RBA-basis from the character table.  Using a computer implementation, we give a list of all noncommutative integral table algebras of rank 6 with orders up to 150.  Four in the list are primitive, but we show these cannot be realized as adjacency algebras of association schemes.  In the last section of the paper we apply our methods to give a precise description of the noncommutative integral table algebras of rank 6 for which the multiplicity of both linear characters is 1.
\end{abstract}

\smallskip
{\small \noindent {\it Key words :} Table algebras, Reality-based algebras, association schemes, character tables.  
} 

\section{Introduction} 

An {\it algebra with involution} over the complex field 
$\mathbb{C}$ is an algebra $A$ over $\mathbb{C}$ that is equipped with a $\mathbb{C}$-conjugate semilinear involution $*$ that satisfies $(\alpha x)^*= \bar{\alpha}x^*$ and $(xy)^*=y^*x^*$ for all $\alpha \in \mathbb{C}$ and $x,y \in A$.  Here $\bar{\alpha}$ denotes the complex conjugate of $\alpha \in \mathbb{C}$.   A {\it reality-based algebra} (or {\it RBA}) $(A,\mathbf{B})$ consists of a finite-dimensional algebra with involution $A$ over $\mathbb{C}$ that has a distinguished basis $\mathbf{B}= \{b_0, b_1, \dots, b_{r-1}\}$ satisfying the following properties: 
\begin{enumerate}
\item $b_0=1_A$ is the multiplicative identity of $A$; 
\item $\mathbf{B}^2 \subseteq \mathbb{R}\mathbf{B}$, in particular all of the structure constants $\lambda_{ijk}$ generated by the basis $\mathbf{B}$ in the expressions $b_i b_j = \sum\limits_{k=0}^{r-1} \lambda_{ijk}b_k$ are real numbers;
\item $\mathbf{B}^* = \mathbf{B}$, so $*$ induces a transposition on the set $\{0,1,\dots,r-1\}$ given by $b_{i^*}=(b_i)^*$ for all $b_i \in \mathbf{B}$; 
\item $\lambda_{ij0} \ne 0$ if and only if $j = i^*$; and 
\item $\lambda_{ii^*0} = \lambda_{i^*i0} > 0$.
\end{enumerate}  
We will refer to the distinguished basis $\mathbf{B}$ of an RBA $(A,\mathbf{B})$ as an RBA-basis.  The dimension $r$ is called the {\it rank} of the RBA.  

Reality-based algebras were introduced by Blau in \cite{B09} in order to collect several generalizations of group algebras appearing in the literature into a single definition.   Examples of RBAs beyond the elementary group algebra $(\mathbb{C}G,G)$ example that are more-or-less well-known include Brauer's pseudogroups, Kawada's $C$-algebras, table algebras (both commutative and noncommutative), and the Bose-Mesner algebras of finite association schemes (both commutative and noncommutative) (see \cite{AF}, \cite{AFM}, \cite{BI84}, and \cite{Z05}).   Table algebras correspond to RBAs for which all structure constants $\lambda_{ijk}$ relative to the RBA-basis are nonnegative.  A table algebra is realized by an association scheme when the RBA-basis can be realized as the set of adjacency matrices of the relations in an association scheme.  

An RBA $(A, \mathbf{B})$ has a {\it positive degree map} when there is a one-dimensional algebra homomorphism $\delta:A \rightarrow \mathbb{C}$ such that $\delta(a^*)=\overline{\delta(a)}$ for all $a \in A$ and $\delta(b_i)>0$ for all $b_i \in \mathbf{B}$.   When the RBA $(A,\mathbf{B})$ has a positive degree map $\delta$, we will say that the distinguished basis $\mathbf{B}$ is an RBA$^{\delta}$-basis.  An RBA$^{\delta}$-basis $\mathbf{B}$ can be positively rescaled to arrange that $\delta(b_i) = \lambda_{ii^*0}$ for all $b_i \in \b$.   This rescaling is achieved by the unique map $b_i \mapsto \dfrac{\delta(b_i)}{\lambda_{ii^*0}} b_i$ for all $b_i \in \mathbf{B}$.  This unique rescaling is called the {\it standard} RBA$^{\delta}$-basis.  
Two standard RBA$^{\delta}$-bases $\mathbf{B}$ and $\mathbf{D}$ of the algebras with involutions
 $A$ and $C$ are in {\it exact isomorphism} when there is a bijection $\pi: \mathbf{B} \rightarrow \mathbf{D}$ for which the linear extension of $\pi$ is an algebra isomorphism.   To classify the RBA-bases of algebras with involutions that have a fixed one-dimensional algebra representation, one needs to distinguish the standard RBA-bases up to exact isomorphism.  

The main results of this article provide a complete characterization of noncommutative 6-dimensional RBAs which admit a positive degree map. We show that every such algebra is isomorphic (as abstract algebra) to $\mathbb{C}\oplus\mathbb{C}\oplus M_2(\mathbb{C})$ and then we
characterize the standard RBA$^{\delta}$-bases of the algebra $A =
\mathbb{C} \oplus \mathbb{C} \oplus M_2(\mathbb{C})$ up to exact isomorphism.  This is a follow-up to a similar characterization by the authors of standard RBA$^{\delta}$-bases of the noncommutative $5$-dimensional semisimple algebra \cite{HMX}.   In Section $2$ we review the character theory of RBAs with positive degree maps, introduce the standard feasible trace, and give character-theoretic formulas for the centrally primitive idempotents of RBAs with positive degree maps.  In Section 3 we characterize the standard RBA$^{\delta}$-bases of the $6$-dimensional noncommutative semisimple algebra $A$.  We show that the standard RBA$^{\delta}$-bases of $A$ are determined by their character tables.  We show that the standard RBA$^{\delta}$-basis is determined by the values of $\delta$ and the other irreducible linear character $\phi$.  We describe the possibilities for these character values, showing the values of $\delta$ are arbitrarily positive, and the values of $\phi$ are freely determined up to a ratio condition.  We obtain a matrix realization of the standard basis elements.  In Section 4 we use a computer implementation to generate a list of all noncommutative integral standard table algebras of rank 6 up to order 150.   Some of the examples we generate are unexpectedly primitive, so in Section 5 we show these do not arise from association schemes.  In Section 6 we describe all noncommutative table algebras of rank 6 for which both irreducible linear characters have multiplicity 1.

\section{Character theory of RBAs} 

Being a finite-dimensional algebra with involution over $\mathbb{C}$, the algebra $A$ appearing in an RBA is always semisimple (see \cite[Theorem 11.2]{Takesaki}).  Also crucial to our discussion is the existence of a {\it feasible trace}, which is a linear map $\mathfrak{t}: A \rightarrow \mathbb{C}$ for which $\mathfrak{t}(xy)=\mathfrak{t}(yx)$ for all $x,y \in A$.  It is easy to see from the RBA definition that  the map given by $\mathfrak{t}(\sum_i \alpha_i b_i) = \alpha_0$, for all elements $\sum_i \alpha_i b_i$ of $A$ expressed in terms of the distinguished basis $\mathbf{B}$ with coefficients in $\mathbb{C}$, is a feasible trace.   

Now suppose $(A, \mathbf{B})$ is an RBA with a positive degree map $\delta$, and suppose $\mathbf{B}$ is the standard basis.  The positive number $n = \sum_i \delta(b_i)$ is called the {\it order} of the RBA $(A,\mathbf{B})$.  A normalization of our initial feasible trace produces the {\it standard feasible trace} of $(A,\mathbf{B})$, which is given by $\tau( \sum_i \alpha_i b_i) = n\alpha_0$, for all $\sum_i \alpha_i b_i \in A$.   The standard feasible trace induces a sesquilinear form on $A$, given by 
$$\langle x, y \rangle = \tau(x y^*), \mbox{ for all } x,y \in A.$$
(Here ``sesquilinear'' means it is a real bilinear form that is $\mathbb{C}$-linear in the first variable and $\mathbb{C}$-conjugate-linear in the second.)  This form is particularly useful for calculating structure constants in an RBA with positive degree map $\delta$, since when the RBA$^\delta$-basis is standard, we have 
\begin{equation}\label{scs1}
\lambda_{ijk} = \dfrac{1}{n\delta(b_k)} \langle b_ib_j,b_k \rangle, \mbox{ for all } 0 \le i,j,k \le r-1. 
\end{equation}
Since $A$ is finite-dimensional and semisimple, any feasible trace can be expressed as a linear combination of the irreducible characters of $A$ (see \cite[Proposition 5.1]{Hig87}).  Let $\irr(A)$ be the set of irreducible characters of $A$, and let $\tau = \sum_{\chi \in \irr(A)} m_{\chi} \chi$ be the expression of the standard feasible trace as a linear combination of irreducible characters of $A$.  The coefficients $m_{\chi}$ in this expression are called the {\it multiplicities} of the RBA.  It is always the case that $m_{\delta}=1$ \cite[Proposition 2.21]{B09} and $m_{\chi}>0$, for all $\chi \in \irr(A)$ (see \cite[Lemma 2.11]{B09}).

Since $A$ is semisimple, we have that for each $\chi \in \irr(A)$, there is a unique centrally primitive idempotent $e_{\chi}$ of $A$ such that $\chi(e_{\chi}) \ne 0$, and for this idempotent we have that $\chi(e_{\chi})=\chi(b_0)$ is a positive integer.   With respect to our $\mathbb{C}$-conjugate-linear involution, the centrally primitive idempotents satisfy $(e_{\chi})^*=e_{\chi}$, for all $\chi \in \irr(A)$ \cite[Proposition 2.9]{B09}.  In particular this implies the simple components of $A$ are $*$-invariant.  Important for us will be the character-theoretic formula that expresses the centrally primitive idempotents $e_{\chi}$ in terms of the RBA$^{\delta}$-basis $\mathbf{B}$, which was first established in this generality in the work of Higman \cite[\S 5]{Hig87}: for all $\chi \in \irr(A)$, 
\begin{equation}\label{cpis}
e_{\chi} = \displaystyle{ \frac{m_{\chi}}{n} \sum_{i=0}^{r-1} \frac{\chi(b_i^*)}{\lambda_{ii^*0}}b_i}.
\end{equation} 
One can see from this formula that $\chi(b_{i^*}) = \overline{\chi(b_i)}$, for all $b_i \in \mathbf{B}$ and $\chi \in \irr(A)$.   When the RBA$^{\delta}$-basis $\mathbf{B}$ is standard, we have $\delta(b_i) = \lambda_{ii^*0}$, for all $i=0,1,\dots,r-1$.  
 Also important for us will be the orthogonality relations: for all $\psi, \chi \in \irr(A)$, 
$$ \psi(e_{\chi}) = 0 \mbox{ unless } \psi = \chi, \mbox{ and in that case } \psi(e_{\psi}) = \psi(b_0). $$

\section{Noncommutative standard RBA$^{\delta}$-bases of rank $6$}
\label{s-exist}

 Let us assume for the moment 
that $(A,\mathbf{B})$ is a noncommutative $6$-dimensional RBA with positive degree map $\delta$, and that $\mathbf{B}=\{b_0=1_A,b_1,b_2,b_3,b_4,b_5\}$ is its standard RBA$^{\delta}$-basis.  Let $\tau$ be the standard feasible trace, so we have from Section 2 that $\tau(x^*x) \ge 0$ for all $x \in A$.   Let $\irr(A)=\{\delta,\phi,\chi\}$, where $\phi$ is the other linear character of $A$ and $\chi$ is the irreducible character of degree $2$.  Again from Section 2 we have that $\tau = \delta + m_{\phi}\phi + m_{\chi}\chi$ with positive multiplicities $m_{\phi}$ and $m_{\chi}$.   We denote the values of $\delta$, $\phi$, and $\chi$ on $\mathbf{B}$ by $\delta(b_i)=\delta_i$, $\phi(b_i)=\phi_i$, and $\chi(b_i)=\chi_i$ for each $b_i \in \mathbf{B}$, for $i=0,1,\dots,5$.  In particular, $\delta_0=1$, $\phi_0=1$, $\chi_0=2$ and $\delta_1,\dots,\delta_5 > 0$.   The values of $\phi$ on $\mathbf{B}$ are necessarily real as the complex conjugate of
$\phi$ is itself, and the values of $\chi$ on $\mathbf{B}$ must also be real as it is the unique irreducible character of degree $2$.   

Since the real subalgebra $\mathbb{R}\mathbf{B}$ is a semisimple $6$-dimensional noncommutative algebra over $\mathbb{R}$, it has a unique noncommutative simple component, which is isomorphic to either $M_2(\mathbb{R})$ or $\mathbb{H}$ (the algebra of real quaternions).  We begin by excluding the latter case. The transpose of a matrix $M$ is denoted by $M^\top$.

\begin{lemma} \label{splitRba} Let $(A,\mathbf{B})$ be a noncommutative reality-based algebra of rank six with an RBA$^\delta$-basis $\mathbf{B}$. Then the algebra ${\mathbb R}\mathbf{B}$ is isomorphic to ${\mathbb R}\oplus\mathbb{R} \oplus M_2({\mathbb R})$ and, up to a change of basis, ${}^*$ acts on $M_2({\mathbb R})$ as matrix transposition.  In particular, $\mathbf{B}$ has exactly four ${}^*$-invariant elements.
\end{lemma}  

\begin{proof}
Since $A$ is non-commutative, the basis $\mathbf{B}$ contains at least one pair $b_i,b_i^*$ of non-symmetric elements. Therefore the number of ${}^*$-fixed elements of ${\mathbf B}$ is either $2$ or $4$. In the first case the dimension of ${}^*$-invariant subspace of $\mathbb{R}\mathbf{B}$ is $4$ while in the second one it is equal to $5$.  
Suppose first that $\mathbb{R}\mathbf{B}\cong {\mathbb R}\oplus\mathbb{R}\oplus \mathbb{H}$.   The noncommutative simple component of $\mathbb{R}\mathbf{B}$ is ${}^*$-invariant, and in particular $1_{\mathbb{H}}^*=1_{\mathbb{H}}$.  Since the dimension of the ${}^*$-invariant subspace of $\mathbb{R}\mathbf{B}$ is at least $4$, there exists a purely imaginary unit quaternion $q \in \mathbb{H}$ such that $q^*= q$.   On the other hand, since $q$ is purely imaginary, $q  = -\bar{q}$.  Therefore for this $q$ we have $q^*q = - 1_{\mathbb{H}}$.  But then $\tau(q^*q)=m_{\chi}\chi(q^*q) = - m_{\chi} \chi(1_{\mathbb{H}}) = - m_{\chi} \chi(1) < 0$, so this is a contradiction.   This excludes the case of $\mathbb{R}\mathbf{B}\cong\mathbb{R}\oplus \mathbb{R}\oplus\mathbb{H}$, so we must have that $\mathbb{R}\mathbf{B} \simeq \mathbb{R}\oplus\mathbb{R} \oplus M_2(\mathbb{R})$.  

Let $\Delta:{\mathbb R}\mathbf{B}\rightarrow M_2({\mathbb R})$ be the two-dimensional irreducible representation of ${\mathbb R}\mathbf{B}$ given by projection to the component $M_2({\mathbb R})$, which affords the character $\chi$. 
We have that $\Delta(x^*)^\top$ is a $2$-dimensional irreducible representation equivalent to $\Delta$.  Thus there exists an $S \in GL_2({\mathbb R})$ such that $\Delta(x^*)^\top = S^{-1} \Delta(x) S$. Equivalently, $\Delta(x^*) = S^\top \Delta(x)^\top (S^{-1})^\top$. Substituting $x^*$ for  $x$ in 
$\Delta(x^*) = S^\top \Delta(x)^\top (S^{-1})^\top$, we obtain that 
$$
\Delta(x) = S^\top \Delta(x^*)^\top (S^{-1})^\top =
(S^\top S^{-1}) \Delta(x) (S^\top S^{-1} )^{-1},
$$
 for each $x\in {\mathbb R}\mathbf{B}$. Combining this together with $\Delta({\mathbb R}\mathbf{B}) = M_2({\mathbb R})$ we obtain that $S^{-1}S^\top = \alpha I_2$ for some $\alpha \in {\mathbb R}$. It follows from $S^\top =\alpha S$ and  $(S^\top)^\top = S$ that $\alpha =\pm 1$, i.e.~$S$ is either symmetric or antisymmetric.

Assume first that $S$ is antisymmetric, that is $S = \begin{bmatrix} 0 & a \\ -a & 0 \end{bmatrix}$. A direct check shows that in this case the map $X\mapsto S X^\top S^{-1}, X\in M_2({\mathbb R})$ has a one-dimensional space of fixed points. So, in this case the dimension of the ${}^*$-invariant subspace of $\mathbb{R}\mathbf{B}$ is three, a contradiction.

Assume now that $S$ is symmetric. Then $S=P^\top D P$ for some $D \in \{I_2, Diag(1,-1),-I_2\}$ and 
$P\in GL_2({\mathbb R})$.  Replacing $\Delta(x)$ by the equivalent representation $\Sigma(x):=(P^{-1})^\top \Delta(x) P^\top$ we obtain  $\Sigma(x^*)= D \Sigma(x)^\top D^{-1}$.  If $D=\pm I_2$, then we are done.  It remains to deny the case of $D=Diag(1,-1)$.  Consider the standard feasible trace $\tau(x)$ for $x\in \mathbb{R}\mathbf{B}$.  Since $\tau(x)=\delta(x)+m_{\phi}\phi(x) + m_{\chi}\chi(x)$ for all $x \in A$, and $\delta(e_{\chi}x)=\phi(e_{\chi}x)=0$ for all $x \in A$, we must have that for all $x \in A$, 
$$ \chi(xx^*) = \chi((e_{\chi}x)(e_{\chi}x)^*) = \frac{1}{m_{\chi}}(\tau((e_{\chi}x)(e_{\chi}x)^*) \ge 0.$$ 
In particular, $tr(\Delta(x)\Delta(x^*)) = \chi(xx^*) \geq 0$ for all $x\in {\mathbb R}\mathbf{B}$.  Since $\Sigma:{\mathbb R}\mathbf{B}\rightarrow M_2({\mathbb R})$ is an epimorphism, we conclude that $tr(X D X^\top D^{-1})\geq 0$ holds for all $X\in M_2(\mathbb{R})$. Now choosing $X=\begin{bmatrix} 0 & 1 \\ 1 & 0 \end{bmatrix}$ we get a contradiction.
\end{proof} 

By Lemma \ref{splitRba}, we can replace $A$ by an isomorphic image $\mathbb{R}\oplus \mathbb{R} \oplus M_2(\mathbb{R})$. Any  standardized RBA$^{\delta}$-basis of $A$ is of the form 
\begin{equation}\label{eq080317}
\begin{array}{rcl} 
\mathbf{B} &=& \{b_0=(1,1,B_0), b_1=(\delta_1,\phi_1,B_1)=b_1^*, b_2=(\delta_2,\phi_2, B_2)=b_2^*, \\
& & \quad b_3 = (\delta_3,\phi_3,B_3)=b_3^*, b_4 = (\delta_4,\phi_4,B_4), b_5 = b_4^*=(\delta_4,\phi_4,B_4^{\top}) \},
\end{array}  
\end{equation}
with $B_0 = I_2$, $B_1$, $B_2$, $B_3$, and $B_4$ all being $2 \times 2$ real matrices.  Here $\delta_i = \delta(b_i)$ for $i = 0, \dots, 4$, where $\delta$ is the positive degree map, and the $\phi_i = \phi(b_i)$ for $i = 0, \dots, 4$ are the values of the other linear character $\phi$ of $A$ on $\mathbf{B}$.  If $\chi$ is the irreducible character of $A$ with degree $2$, then its values on $\mathbf{B}$ are $\chi_i:=\chi(b_i) = tr(B_i)$ for $i=1,2,3,4$, and the map $\Delta$
given by $\Delta(b_i)= B_i$ for $i=0,\dots,4$ and $\Delta(b_5)=B_4^{\top}$ defines a real representation of $A$ that affords $\chi$.

The goal of this section is to give necessary and sufficient conditions for two lists of real numbers $[\delta, \phi] = [(\delta_1,\dots,\delta_4),(\phi_1,\dots,\phi_4)]$ to be the values of the degree map and the other linear character of a noncommutative rank $6$ RBA.   We will show that, under these conditions, the two lists of values $[\delta, \phi] = [(\delta_1,\dots,\delta_4),(\phi_1,\dots,\phi_4)]$ completely determine the RBA up to a basis permutation and a possible conjugation of the $2 \times 2$ real matrices $B_i$ by an orthogonal matrix. 

Recall that the standard feasible trace $\tau$ of $A$ induces the sesquilinear form $\langle x , y \rangle = \tau(xy^*)$, for all $x,y \in A$.  Since 
$\tau = \delta + m_{\phi}\phi + m_{\chi}\chi$, this form can be expressed as 
\begin{equation}
\label{e-form}
\langle x , y \rangle = \tau(xy^*) = \delta(x)\delta(y^*) + m_{\phi}\phi(x) \phi(y^*) + m_{\chi}\chi(xy^*).
\end{equation}
Here $\chi(xy^*)$ agrees with the usual trace bilinear form on real $2 \times 2$ matrices. That is, 
if $X$ and $Y$ are the images of $x$ and $y$ under a representation affording $\chi$, then $\chi(xy^*) = tr(X \overline{Y}^{\top})$.   

By \cite[Formula (7.4)]{Hig87} the feasible multiplicity $m_\psi$ of an irreducible character $\psi$ 
may be computed via $\sum_{i}\frac{|\psi(b_i)|^2}{\delta_i} = \psi(1)\frac{n}{m_\psi}$.
Therefore, $m_\delta=1, \sum_{i}\frac{\phi_i^2}{\delta_i} = \frac{n}{m_\phi}, 
\sum_{i}\frac{\chi_i^2}{\delta_i} = 2\frac{n}{m_\chi}$. It follows from $n=\tau(b_0)$ that $n=1+m_\phi+2m_\chi.$

Let $\mathbf{E}=\{e_0,...,e_5\}$ denote a ``standard" basis of the algebra $\R\oplus\R\oplus M_2(\R)$:
$$
\begin{array}{rclrclrcl}
e_0 & = & (1,0,O_2), \qquad & e_1 & = & (0,1,O_2),\qquad & e_2 & = & (0,0,E_{11}), \\
e_3 & = & (0,0,E_{22}), & e_4 & = & (0,0,E_{12}), \mbox{ and }& e_5 & = & (0,0,E_{21}),
\end{array}
$$
where $E_{ij}$ are the standard matrix units of $M_2(\R)$.   It follows from the definition of $\mathbf{E}$ that each of the numbers $\delta(e_ie_j^*),\phi(e_i e_j^*),\chi(e_i e_j^*)$ is zero whenever $i\neq j$. Therefore $\mathbf{E}$ is an orthogonal basis with respect to the form $\langle \ , \ \rangle$
defined in \eqref{e-form}. 
 Notice that $\b$ is an orthogonal basis with respect to the same form.

Starting from now we assume that $\b=\{b_0,...,b_5\}$ is an arbitrary basis of the algebra $A=\mathbb{R}\oplus \mathbb{R}\oplus M_2(\mathbb{R})$ which has form~\eqref{eq080317}.  In what follows we denote by $P(\b)$ the transition matrix between the bases $\mathbf{E}$ and $\b$. That is, the $j$-th column of $P(\b)$ 
consists of the coordinates of $b_j$ in the basis $\mathbf{E}$:
$$
b_j = \sum_{i=0}^5 P(\b)_{ij} e_i, \quad
\hbox{where $P(\b)_{ij}$ is the $(i, j)$-entry of } P(\b).
$$ 
Following the usual notation (for example, see \cite{BI84}),
 we denote  by $Q(\b)$ the matrix that satisfies $P(\b)Q(\b) = n I_6$,
 where $n =\sum_{i}\delta_i \ne 0$. Thus
$$
e_j = n^{-1}\sum_{i=0}^5 Q(\b)_{ij} b_i.
$$ 
If $\b$ is a standardized RBA${}^{\delta}$-basis, then it is an orthogonal basis with respect to the bilinear form $\langle \ , \ \rangle$ defined in
\eqref{e-form}.  Computing $\langle b_i,e_j\rangle$ in two ways we obtain the following formula 
\begin{equation}
\frac{Q(\b)_{ij}}{m_j} = \frac{P(\b)_{ji}}{\delta_i}
\end{equation}\label{eq:PQ}
where $m_i:=\langle e_i,e_i\rangle$.

Let $\si: =(\si_0, \si_1, \dotsc, \si_5) \in \mbr^6$, where each $\si_i >
0$. For $u = (u_0, u_1, \dotsc, u_5), v = (v_0, v_1, \dotsc, v_5) \in
\mbr^6$, define
\begin{equation}
\label{e-inner}
(u, v)_\si := \sum_{i} u_i v_i \si_i^{-1}.
\end{equation}
Then $( \ , \ )_\si$ is an inner product on $\mbr^6$.

{\bf Remark.} Although the ``only if'' part of the theorem below may be directly deduced from Higman's paper~\cite{Hig87}, we prefer to give a direct proof here to make the paper self-contained.

\begin{thm}\label{010616a}
Let $\b:=\{b_i = (\delta_i,\phi_i,B_i)\,|\,i=0,...,5\}$ be an arbitrary basis of the algebra 
$A:=\mathbb{R}\oplus \mathbb{R} \oplus M_2(\mathbb{R})$, and 
 $P(\b)$ the transition matrix between the bases $\mathbf{E}$ and $\b$. 
Then $\b$ is an RBA${}^\delta$-basis of $A$ with respect to the involution $(x,y,Z)^* = (x,y,Z^\top)$ and the degree map $\delta((x,y,Z))=x$ if and only if the following conditions are satisfied (up to the renumbering of the elements of $\b$).
\begin{enumerate}
\item $\delta_i > 0,i=0,...,5$ and $\delta_0=\phi_0=1,B_0 = I_2$.

\item 
$ \phi_5=\phi_4, \delta_5=\delta_4$, $B_i^T=B_i,i=1,2,3$, and $B_4^T = B_5$.

\item 
Let $n:=\sum_{i}  \delta_i, \  m_\delta =1, \  m_\phi = \frac{n}{\sum_{i} (\phi_i^2/\delta_i)}$, 
and $m_\chi = \frac{n-1-m_\phi}{2}$. Then 
$$
m_\chi > 0 \quad \hbox{and} \quad
\left( P(\b)_i, P(\b)_j \right)_\delta = \frac{n}{m_i} \delta_{ij}, \quad 0 \le i, j \le 5, 
$$
where $P(\b)_i$ is the $i$-th row of $P(\b)$, $0 \le i \le 5$, 
$\delta = (\delta_0, \delta_1, \dotsc, \delta_5)$ (by abuse of the notation), 
$\delta_{ij}$ is the Kronecker delta, and $m_0 = m_\delta, m_1 = m_\phi$, 
$m_2 = \dotsb = m_5 = m_\chi$.  
\end{enumerate}
\end{thm}

\begin{proof}  First, suppose $\b$ is an RBA${}^\delta$-basis  of $A$ with  $b_0 = 1_A$. Since  the degree map $\delta$ of $A$ is the projection on the first coordinate, the numbers $\delta_i$ are  positive. Also $b_0=1_A$ implies $\delta_0=\phi_0 =1, B_0=I_2$. Thus {\it (i)} holds. 

It follows from Lemma~\ref{splitRba} that $\b$ contains only two non-symmetric elements. Without loss of generality we may assume that they are $b_4$ and $b_5$. Thus $b_1^* = b_1,b_2^*=b_2,b_3^*=b_3,b_4^*=b_5,b_5^*=b_4$, and {\it (ii)} holds.

Since $\mathbf{E}$ is an orthogonal basis, we can write 
$$
\delta_{\ell k} m_k = \langle e_\ell,e_k\rangle = n^{-2}\sum_{i=0}^5 Q(\b)_{i\ell} Q(\b)_{ik}\langle b_i,b_i\rangle. 
$$
Using \eqref{eq:PQ} we obtain that
$$ 
\delta_{\ell k} m_k =
n^{-2}\sum_{i=0}^5 \frac{m_\ell}{\delta_i} P(\b)_{\ell i} \frac{m_k}{\delta_i} P(\b)_{k i}(n\delta_i).
$$
Therefore,
$$
  m_\ell \sum_{i=0}^5  P(\b)_{\ell i} P(\b)_{k i}\delta_i^{-1}= n \delta_{\ell k}.
$$
That is, 
$$
(P(\b)_\ell, P(\b)_k)_\delta = \frac{n}{m_\ell} \delta_{\ell k}, \quad 0 \le \ell, k \le 5.
$$
Note that $\tau = m_\delta \delta + m_\phi \phi + m_\chi \chi$, and $m_\chi > 0$ 
(cf. \cite{Hig87}). 
It follows from $e_0=e_\delta$ that $m_0=\langle e_\delta,e_\delta\rangle = \tau(e_\delta)=1=m_\delta$.  The equality $e_1=e_\phi$ implies that  $m_1=\langle e_\phi,e_\phi\rangle = \tau(e_\phi)  = m_\phi$.  For each $2\leq i\leq 5$ we have $m_i = \langle e_i,e_i\rangle = \tau(e_ie_i^*)=m_\chi \chi(e_i e_i^*)=m_\chi$.  This proves (iii).

Assume now that the basis $\b$ satisfies the assumptions (i)-(iii) of the theorem. 
Parts (i)-(ii) of the assumptions imply that the basis $\b$ satisfies the first three axioms of an RBA$^\delta$-basis.  The third assumption implies that  $P:=P(\b)$ satisfies the matrix equation $M P \Delta^{-1} P^T = nI_6$, where $\Delta$ and $M$ are diagonal matrices defined via $\Delta_{ii}=\delta_i, M_{ii}=m_i$. 
Therefore $P^T M P \Delta^{-1} = n I_6$, and consequently,  $P^T M P = n \Delta$.
Let ${P_i}^\top$ be the $i$-th row of $P^\top$, $0 \le i \le 5$. Then, 
$$
\left ({P_i}^\top, {P_j}^\top \right)_\si = n \delta_i \delta_{ij}, \quad
0 \le i, j \le 5,
$$
where $\si := (m_0^{-1}, m_1^{-1}, \dotsc, m_5^{-1})$. Note that
\begin{eqnarray*}
b_i b_j^* & = & P_{0i}P_{0j}e_0 + P_{1i}P_{1j}e_1 + (P_{2i}P_{2j} + P_{4i}P_{4j})e_2 +
(P_{5i}P_{5j} + P_{3i}P_{3j})e_3 \\
& & \ + (P_{2i}P_{5j} + P_{4i}P_{3j})e_4 +
(P_{5i}P_{2j} + P_{3i}P_{4j})e_5.
\end{eqnarray*}
Since $P_0^\top = (1, 1, 1, 1, 0, 0)$, $m_2 = m_3 = m_4 = m_5$, and 
$b_i b_j^* = \sum_{k} \l_{ij^*k} b_k$, we get that
$$
\left ({P_i}^\top, {P_j}^\top \right)_\si =
\sum_{k} \l_{ij^*k} \left( P_k^\top, {P_0}^\top \right)_\si =
\l_{ij^*0} \left( P_0^\top, {P_0}^\top \right)_\si =  n \l_{ij^*0}.
$$
Thus, $\l_{ij^*0} = \delta_i \delta_{ij}$, $0 \le i, j \le 5$. Furthermore, since
$\delta_i = \delta_{i^*}$, we have $\l_{ii^*0}= \l_{i^* i0} = \delta_i > 0$. So 
$\b$ is an RBA$^\delta$-basis.
\end{proof}

Denoting $B_i$ by $\begin{bmatrix} r_i & s_i \\ t_i  & u_i \end{bmatrix}$,
 we can write the transition matrix $P:=P(\b)$ in the following form
\begin{equation}\label{pmatrix}
\begin{array}{c|cccccc}
\ & b_0 & b_1 & b_2 & b_3 & b_4 & b_5\\
\hline
e_0 & 1 & \delta_1 & \delta_2 & \delta_3 & \delta_4 & \delta_4\\
e_1 & 1 & \phi_1 & \phi_2 & \phi_3 & \phi_4 & \phi_4\\
e_2 & 1 & r_1 & r_2 & r_3 & r_4 & r_4\\
e_3 & 1 & u_1 & u_2 & u_3 & u_4 & u_4\\
e_4 & 0 & s_1 & s_2 & s_3 & s_4 & t_4\\
e_5 & 0 & s_1 & s_2 & s_3 & t_4 & s_4\\
\end{array}
\end{equation}
Notice that $e_\delta = e_0,e_\phi=e_1$ and $e_\chi=e_2+e_3$. It follows from the row orthogonality (Theorem~\ref{010616a}(iii)) that the sum of each row of $P$ (except for the first row) is zero. In particular $1+\phi_1 + \phi_2 +  \phi_3 + 2\phi_4=0$.

We claim that it is impossible to have $\frac{\phi_i}{\delta_i} = \frac{\phi_4}{\delta_4}$ for all $i=1,2,3$.  If these ratios were all equal to the same constant $\lambda$, then the orthogonality of the second and the third rows of 
$P$ would imply that 
$$
0 =1 +\sum_{i=1}^5 \phi_i r_i\delta_i^{-1} = 1+\lambda\sum_{i=1}^5 r_i = 1-\lambda.
$$   
Thus, $\lambda = 1$. But then all of the $\phi_i$ would be positive, a contradiction to that $1+\sum_{i=1}^5 \phi_i = 0$.  

Now we are ready to prove the main result of the section. 

\begin{thm}\label{020616a}
Let $\delta_1,...,\delta_4$ and $\phi_1,...,\phi_4$  be real numbers such that
 $\delta_i > 0$ for all $i=1,\dots,4$ and $1 + \phi_1+ \phi_2+\phi_3 + 2\phi_4 
= 0$. Assume that $\delta_0 = 1$, $\delta_5 = \delta_4$, $\phi_0 = 1$, 
$\phi_5 = \phi_4$, and   
$\phi_4/\delta_4\neq\phi_i/\delta_i$ for some $i\in\{1,2,3\}$. Then there exist real matrices 
$$
B_0=I_2,B_1=\begin{bmatrix} r_1 & s_1 \\ s_1  & u_1 \end{bmatrix},
B_2=\begin{bmatrix} r_2 & s_2 \\ s_2  & u_2 \end{bmatrix},
B_3=\begin{bmatrix} r_3 & s_3 \\ s_3  & u_3 \end{bmatrix},
B_4=\begin{bmatrix} r_4 & s_4 \\ t_4  & u_4 \end{bmatrix}
,B_5:=B_4^\top
$$ 
such that the basis $\b=\left\{b_i:=(\delta_i,\phi_i,B_i) \mid 0 \le i \le 5
\right\}$ is an RBA$^{\delta}$-basis of $A$.

Furthermore, the matrices $B_0,B_1,...,B_5$ are unique up to a conjugation 
by an orthogonal  $2 \times 2$ matrix and a permutation of the indices of 
$B_i$ by an element of the group 
$\langle (1,2), (1,2,3) \rangle \times \langle (4,5) \rangle$.
\end{thm}

\begin{proof} First of all, let us fix the following numbers:
$$
n :=\sum_{i=0}^5\delta_i, \ \
 m_\phi :=\frac{n}{\sum_{i=0}^5 (\phi_i^2 / \delta_i)}, \ \ 
m_\chi:=\frac{n-1-m_\phi}{2}.
$$
Without loss of generality, for the rest of the proof we assume that 
$\frac{\phi_3}{\delta_3}\neq\frac{\phi_4}{\delta_4}$.

In what follows we will need the inequality $m_\chi > 0$. For this reason let
us first prove this inequality. It is clear that $m_\chi > 0$ if and only
if $m_\phi < n - 1$, which is equivalent to 
$n < (n -1) \sum_{i=0}^5 (\phi_i^2 / \delta_i)$. Since $\delta_0 = \phi_0 
=1$, we see that
$$
m_\chi > 0 \quad \hbox{if and only if} \quad
 \sum_{i=1}^5 \frac{\phi_i^2}{\delta_i} > \frac{1}{n-1}.
$$
The expression $\sum_{i=1}^5 \frac{\phi_i^2}{\delta_i}$ is a function of the real variables $\phi_i$ which satisfy $\phi_1 + ... +\phi_5 = -1$.  Using the Lagrange multipliers, one can find that the minimal value of 
$\sum_{i=1}^5 \frac{\phi_i^2}{\delta_i}$ is $\frac{1}{n-1}$, and this value
 is reached when the ratios $\phi_i/\delta_i,i=1,...,5$ are all equal to $-\frac{1}{n-1}$. 
 Thus, $\sum_{i=1}^5 \frac{\phi_i^2}{\delta_i}\geq \frac{1}{n-1}$, and the equality holds only when $\phi_i/\delta_i = -\frac{1}{n-1}$, $i=1,...,5$.   
But by the assumption, the equality
$\phi_i/\delta_i = -\frac{1}{n-1}$ is not satisfied for all $i\in\{1,...,5\}$.  Thus, $\sum_{i=1}^5 \frac{\phi_i^2}{\delta_i} > \frac{1}{n-1}$, and
hence $m_\chi > 0$, as desired.

To prove the existence of matrices $B_i, 1 \le i \le 4$, such that 
$\b=\left\{b_i:=(\delta_i,\phi_i,B_i) \mid 0 \le i \le 5
\right\}$ is an RBA$^{\delta}$-basis of $A$, it is equivalent to prove  that the transition matrix 
$$
P:=P(\b)= \left(
\begin{array}{cccccc}
1 & \delta_1 & \delta_2 & \delta_3 & \delta_4 & \delta_4\\
 1 & \phi_1 & \phi_2 & \phi_3 & \phi_4 & \phi_4\\
 1 & r_1 & r_2 & r_3 & r_4 & r_4\\
 1 & u_1 & u_2 & u_3 & u_4 & u_4\\
 0 & s_1 & s_2 & s_3 & s_4 & t_4\\
 0 & s_1 & s_2 & s_3 & t_4 & s_4\\
\end{array}
\right)
$$ 
satisfies the condition (iii) of Theorem~\ref{010616a}.  Since any symmetric matrix is conjugate to a diagonal matrix by an orthogonal matrix, we may assume that one of $B_1,B_2,B_3$ is diagonal.  So for the rest of the proof, we assume that $B_1$ is diagonal. That is, $s_1=0$. 

By our assumptions the first two rows of $P$ do satisfy the condition (iii) of Theorem~\ref{010616a}.  It remains to find the other rows of $P$ that will satisfy the condition (iii) of Theorem~\ref{010616a}.

Let $P_i$ be the $i$-th row of $P$, $0 \le i \le 5$, and let 
$$
P':=\left(
\begin{array}{c}
P_0\\
P_1\\
P_2+P_3\\
P_2 - P_3\\
P_4 + P_5\\
P_4 - P_5
\end{array}
\right).
$$
Let $P'_i$ be the $i$-th row of $P'$, $0 \le i \le 5$.  It is easy to see that $P$ satisfies the condition (iii) of Theorem~\ref{010616a} if and only if $(P'_i, P'_j)_\delta = 0$ for all $i \ne j$, and 
$$
(P'_0, P'_0)_\delta = n, \ (P'_1,P'_1)_\delta =\frac{n}{m_\phi}, \
(P'_2,P'_2)_\delta = (P'_3,P'_3)_\delta = (P'_4,P'_4)_\delta = 
(P'_5,P'_5)_\delta = 2\frac{n}{m_\chi},
$$
where by abuse of the notation, $\delta = (\delta_0, \delta_1, \dotsc, 
\delta_5)$.

Since $P'_5 = (0,0,0,0,s_4-t_4,t_4-s_4)$,  $(P'_5,P'_5)_\delta = 
2\frac{n}{m_\chi}$ if and only if $P'_5 =\pm \sqrt{\frac{n\delta_4}{m_\chi}} (0,0,0,0,1,-1)$.
The choice of the sign is equivalent to the permutation 
$b_4\leftrightarrow b_5$.  

Thus it remains to find the third, fourth and fifth rows of $P'$. We take 
$$
P_2+P_3 := \frac{1}{m_\chi}\left(
(n,0,0,0,0,0) - m_\phi P_1 - P_0\right).
$$   
(This choice follows from the equality $b_0 = e_\delta +e_\phi + e_\chi$).  Direct computation shows that  
$$
(P_2+P_3,P_0)_\delta = (P_2+P_3,P_1)_\delta =0, \quad
\hbox{and} \quad (P_2+P_3,P_2+P_3)_\delta = 2\frac{n}{m_\chi}. 
$$
Notice that the choice for $P_2+P_3$ is unique, because its entries are the values of $\chi$.

Since the row $P_4+P_5=(0,0,2s_2,2s_3,s_4+t_4,s_4+t_4)$ is orthogonal to 
both $P_0$ and $P_1$, we see that 
$$
  s_2 +  s_3 + (s_4+t_4) = 0 \quad \mbox{ and } \quad 
\frac{\phi_2}{\delta_2} s_2 + \frac{\phi_3}{\delta_3} s_3 +  \frac{\phi_4}{\delta_4} (s_4+t_4) = 0.
$$
But $\phi_3/\delta_3\neq\phi_4/\delta_4$. So the above system has rank two, and therefore has a one-dimensional solution space spanned by the vector  
$$
w=\left(0,0,\frac{\phi_4}{\delta_4}-\frac{\phi_3}{\delta_3},
\frac{\phi_2}{\delta_2}-\frac{\phi_4}{\delta_4},
\frac{\phi_3}{\delta_3}-\frac{\phi_2}{\delta_2},\frac{\phi_3}{\delta_3}-\frac{\phi_2}{\delta_2}\right).
$$
To satisfy the condition $(P_4+P_5,P_4+P_5)_\delta =2\frac{n}{m_\chi}$, one has to set $P_4+P_5 := \lambda w$, where $\lambda = \pm\sqrt{\frac{2n}{m_\chi (w,w)_\delta}}$. The choice of the sign is free, because a different choice of the sign
 is  equivalent to the conjugation of $B_i$'s by the matrix $
\begin{bmatrix} -1 & 0 \\ 0  & 1 \end{bmatrix}$.
Notice that the above $P_4+P_5$ is also orthogonal to 
$P_2+P_3$, because $P_2+P_3$ is a linear combination of $(1,0,0,0,0,0)$, $P_0$ and $P_1$. 

It remains to find $P_2-P_3 = (0,r_1-u_1,r_2-u_2,r_3-u_3,r_4-u_4,r_4-u_4)$. 
From the orthogonality conditions of the rows of $P'$ with respect to the 
inner product $(\ , \ )_\delta$, 
we obtain the following system of equations:
$$
\left\{
\begin{array}{c}
(r_1-u_1)+(r_2-u_2)+(r_3-u_3)+2(r_4-u_4) = 0\\
\frac{\phi_1}{\delta_1}(r_1-u_1)+\frac{\phi_2}{\delta_2}(r_2-u_2)+\frac{\phi_3}{\delta_3}(r_3-u_3)+2\frac{\phi_4}{\delta_4}(r_4-u_4) = 0\\
\left(\frac{\phi_4}{\delta_4}- \frac{\phi_3}{\delta_3}\right)(r_2-u_2)+\left(\frac{\phi_2}{\delta_2}- \frac{\phi_4}{\delta_4}\right)(r_3-u_3)+2\left(\frac{\phi_3}{\delta_3}- \frac{\phi_2}{\delta_2}\right)(r_4-u_4) = 0\\
\end{array}
\right.
$$
This system has rank three. Therefore, up to a sign there exists a unique solution $u$ with 
$(u,u)_\delta =2\frac{n}{m_\chi}$. We can chose either one to be our $P_2-P_3$. As before, a different choice
of the sign is equivalent to a conjugation by 
the matrix $\begin{bmatrix} 0 & 1 \\ 1  & 0 \end{bmatrix}$. 
It is clear now that the above constructed matrix $P'$ also satisfies the property 
$(P'_i, P'_j)_\delta = 0$ for all $i \ne j$. So the theorem holds.
\end{proof}

\medskip
A consequence of Theorem~\ref{020616a} is that every possible character table of a noncommutative rank $6$ RBA from Theorem~\ref{010616a} actually occurs.  

\begin{corollary}\label{chartable}
The possible character tables of noncommutative rank $6$ RBAs with positive degree map $\delta$ are precisely those of the form 
$$\begin{array}{r|cccccc|l} 
& b_0 & b_1 & b_2 & b_3 & b_4 & b_5 & \mbox{ multiplicity } \\ \hline
\delta & 1 & \delta_1 & \delta_2 & \delta_3 & \delta_4 & \delta_4 & 1 \\
\phi & 1 & \phi_1 & \phi_2 & \phi_3 & \phi_4 & \phi_4 & m_{\phi} \\
\chi & 2 & \chi_1 & \chi_2 & \chi_3 & \chi_4 &\chi_4 & m_{\chi}
\end{array}$$
where 

$\delta_i > 0$ for $i=1,\dots,4$, 

\smallskip
$1 + \phi_1 + \phi_2 + \phi_3 + 2 \phi_4 = 0$, \  $\dfrac{\phi_i}{\delta_i} \neq \dfrac{\phi_4}{\delta_4}$ for some $1\leq i \leq 3$, 

\smallskip
$m_\phi = \dfrac{n}{\sum_{i}(\phi_i^2/\delta_i)}$ and $\ m_\chi=\frac{n-1-m_\phi}{2}$, where $n = 1 + \delta_1 + \delta_2 + \delta_3 + 2\delta_4$, and 

\smallskip
$\chi_i = \dfrac{-\delta_i-m_{\phi}\phi_i}{m_{\chi}}$, for $i=1,2,3,4$.

Conversely, any such character table determines a noncommutative rank $6$ RBA up to exact isomorphism. 
\end{corollary} 

\begin{proof} 
This follows from Theorem \ref{010616a}, since we now know that every such set of parameters $[(\delta_i),(\phi_i)]$ determines the standard basis of a noncommutative rank $6$ RBA up to exact isomorphism.  To compute the values of $\chi$ we use the fact that $0 = \tau(b_i) = \delta_i + m_{\phi}\phi_i +m_{\chi}\chi_i$ for $i=1,2,3,4$. 
\end{proof} 

In the above corollary, if the algebra
 is a standard table algebra, then $|\phi_i|\leq \delta_i$ and $|\chi_i|\leq 2\delta_i$.

\section{Examples} 
\label{s-ex}

Using the results in Section \ref{s-exist}, we have been able to use a computer to enumerate noncommutative integral RBA$^{\delta}$-bases of rank $6$ of order $n \le 150$. 
Since our main interest focuses on standard integral table algebras, we restricted the parameters by the inequality  $|\phi_i| \le \delta_i, i=1,\dots,4$ (see \cite[Proposition 4.1]{x08}), so our list includes all non-commutative standard integral table algebras of rank $6$ up to order $150$.   To achieve an enumeration for order $n$, we first find all feasible character tables with integer entries.  For each character table, we compute a standard RBA$^{\delta}$-basis using the procedure explained in Theorem \ref{010616a}, then check for integrality of structure constants using (\ref{scs1}).  

 In this table, $C_m$ represents a cyclic group (thin scheme) of order $m$, $K_m$ represents a rank
$2$ scheme of order $m$, $U_m$ and $T_m$ represent symmetric and anti-symmetric rank $3$ schemes of order $m$, respectively.  $D_m$ denotes a table algebra of rank $2$ having non-integral rational order $m$, and $E_m$ is a noncommutative rank $5$ RBA of order $m$.  Extensions of a normal closed subset $U$ are indicated by $U \rtimes T$ or $U:T$ depending on whether or not the extension splits.  Almost always the extensions of a table algebra by a table algebra results in a table algebra, but there are a few exceptions.  Only in a few circumstances do we know when a table algebra that is an extension of two schemes is realized by a scheme.  One new imprimitive family we encountered is a circle product of a noncommutative rank $5$ RBA with a rank $2$ scheme, which we label with $E \circ K$.  As the order increases these occur with increasing frequency, and so in the table only those $E \circ K$'s with order up to $50$ are listed.  In some cases the RBA or table algebra our construction produces is primitive.  Such examples were not noted previously.

\bigskip
\centerline{ {\bf Table 1: Parameter Sets for Rank 6 RBAs of order $\le$ 150 with $|\phi_i| \le \delta_i$.} }

{\small
$$
\begin{array}{r|l|l|l}
n & [\delta,\phi] & (m_{\phi},m_{\chi}) & \mbox{comments} \\ \hline
6 & [(1,1,1,1),(-1,-1,-1,1)] & (1,2) & \mbox{The group $S_3 \simeq C_3 \rtimes C_2$} \\
10 & [(1,2,2,2),(-1,2,2,-2)] & (1,4) & \mbox{$U_5 \rtimes C_2$, {\tt as10-10}} \\
14 & [(1,3,3,3),(-1,-3,-3,3)] & (1,6) & \mbox{$T_7 \rtimes C_2$, {\tt as14-10}} \\
15 & [(4,4,4,1),(-1,-1,-1,1)] & (4,5) & \mbox{$C_3:K_5$, $3$-array, {\tt as15-16}} \\
18 & [(1,4,4,4),(-1,4,4,-4)] & (1,8) & \mbox{$U_9 \rtimes C_2$, {\tt as18-37}} \\
21 & [(2,2,8,4),(-1,-1,-1,1)] & (8,6) & \mbox{$CS(b_1,b_2)=PG(1,2)$, {\tt as21-19}} \\
22 & [(1,5,5,5),(-1,-5,-5,5)] & (1,10) & \mbox{$T_{11} \rtimes C_2$, {\tt as22-8}} \\
24 & [(7,7,7,1),(-1,-1,-1,1)] & (7,8) & \mbox{$C_3:K_8$, {\tt as24-72}} \\
26 & [(1,6,6,6),(-1,6,6,-6)] & (1,12) & \mbox{$U_{13} \rtimes C_2$, {\tt as26-21}} \\
30 & [(1,7,7,7),(-1,-7,-7,7)] & (1,14) & \mbox{$T_{15} \rtimes C_2$, {\tt as30-73}} \\
33 & [(10,10,10,1),(-1,-1,-1,1)] & (10,11) & \mbox{$C_3:K_{11}$, TA not AS} \\
34 & [(1,8,8,8),(-1,8,8,-8)] & (1,16) & \mbox{$U_{17} \rtimes C_2$, {\tt as34-9}} \\
35 & [(2,2,6,12),(2,2,-1,-2)] & (6,14) & \mbox{$U_5:K_7$, TA not AS} \\
38 & [(1,9,9,9),(-1,-9,-9,9)] & (1,18) & \mbox{$T_{19} \rtimes C_2$, {\tt as38-19}} \\
42 & [(1,10,10,10),(-1,10,10,-10)] & (1,20) & \mbox{$U_{21} \rtimes C_2$} \\
42 & [(6,10,15,5),(-1,0,0,0)] & (36,5/2) & \mbox{$E_6 \circ K_7$, Not TA} \\
42 & [(13,13,13,1),(-1,-1,-1,1)] & (13,14) & \mbox{$C_3:K_{14}$, TA} \\
45 & [(4,4,4,16),(-1,0,0,0)] & (36,4) & \mbox{$E_9 \circ K_5$, Not TA} \\
45 & [(4,12,12,8),(-1,0,0,0)] & (36,4) & \mbox{$E_9 \circ K_5$, Not TA} \\
46 & [(1,11,11,11),(-1,-11,-11,11)] & (1,22) & \mbox{$T_{23} \rtimes C_2$} \\
48 & [(4,12,15,8),(0,0,-1,0)] & (45,1) & \mbox{$E_{3} \circ K_{16}$, Not TA} \\
50 & [(1,12,12,12),(-1,12,12,-12)] & (1,24)  &\mbox{$U_{25} \rtimes C_2$} \\
50 & [(5,10,24,5),(0,0,-1,0)] & (48,\frac12) & \mbox{$E_2 \circ K_{25}$, Not TA} \\
51 & [16,16,16,1],[-1,-1,-1,1] & (16,17) & \mbox{$C_3:K_{17}$, TA} \\
52 & [3,3,27,9],[-1,-1,-1,1] & (27,12) & \mbox{$CS(b_1,b_2)=PG(1,3)$} \\
54 & [1,13,13,13],[-1,-13,-13,13] & (1,26) & \mbox{$T_{27} \rtimes C_2$, TA} \\
58 & [1,14,14,14],[-1,14,14,-14] & (1,28) &  \mbox{$U_{29} \rtimes C_2$} \\
60 & [2,2,11,22],[2,2,-1,-2] & (11,24) & \mbox{$U_5:K_{12}$, TA} \\
60 & [19,19,19,1],[-1,-1,-1,1] & (19,20) & \mbox{$C_3:K_{20}$, TA} \\
62 & [1,15,15,15],[-1,-15,-15,15] & (1,30) & \mbox{$T_{31} \rtimes C_2$} \\
63 & [8,24,24,3],[-1,-3,-3,3] & (8,27) & \mbox{$T_7:K_9$, $7$-array AS} \\
64 & [7,7,35,7],[-1,-1,3,-1] & (35,14) & \mbox{Not TA, primitive} \\
64 & [7,14,14,14],[7,-2,-2,-2] & (7,28) & \mbox{$K_8:T_8$, TA} \\
64 & [9,9,27,9],[1,1,-5,1] & (27,18) & \mbox{Not TA, primitive} \\
66 & [1,16,16,16],[-1,-16,-16,16] & (1,32) & \mbox{$U_{33} \rtimes C_2$}  \\
66 & [5,15,15,15],[-1,1,1,-1] & (45,10) & \mbox{$K_6:T_{11}$, Not TA} \\
69 & [ 22, 22, 22, 1 ], [ -1, -1, -1, 1 ] & ( 22, 23 ) & \mbox{$C_3:K_{23}$, TA} \\
70 & [ 1, 17, 17, 17 ], [ -1, -17, -17, 17 ] & (1, 34) & \mbox{$T_{35} \rtimes C_2$}  \\
\end{array}$$}

{\small
$$\begin{array}{r|l|l|l}
n & [\delta,\phi] & (m_{\phi},m_{\chi}) & \mbox{comments} \\ \hline
74 & [ 1, 18, 18, 18 ], [ -1, 18, 18, -18 ] & (1, 36) &  \mbox{$U_{37} \rtimes C_2$}  \\
78 & [ (1, 19, 19, 19 ), ( -1, -19, -19, 19 )] & (1, 38) & \mbox{$T_{39} \rtimes C_2$} \\
78 & [ (25, 25, 25, 1), ( -1, -1, -1, 1) ] & (25, 26) & \mbox{$C_3:K_{26}$, TA}  \\
81 & [ (8, 8, 32, 16 ), ( -1, -1, 5, -2 )] & (32, 24) & \mbox{Not TA, primitive} \\
81 & [ (10, 10, 20, 20 ), ( 1, 1, -7, 2) ] & (20, 30) & \mbox{TA, primitive} \\
82 &  [ (1, 20, 20, 20 ), ( -1, 20, 20, -20) ] & (1, 40) &  \mbox{$U_{41} \rtimes C_2$} \\
85 &  [ (2, 2, 16, 32 ), ( 2, 2, -1, -2) ] & (16, 34) & \mbox{$K_{17} : T_5$, TA} \\
86 &  [ (1, 21, 21, 21), ( -1, -21, -21, 21) ] & (1, 42) & \mbox{$T_{43} \rtimes C_2$} \\
87 & [ (28, 28, 28, 1), ( -1, -1, -1, 1 )]  & (28, 29) & \mbox{$C_3:K_{29}$, TA}  \\
90 & [ (1, 22, 22, 22), ( -1, 22, 22, -22) ] & (1, 44) & \mbox{$U_{45} \rtimes C_2$}  \\ 
91 & [ (10, 10, 30, 20), ( 10, 10, -9, -6 )]  & (10/3,130/3) & \mbox{$U_{21}:D_{13/3}$, TA} \\
94 &  [ (1, 23, 23, 23), ( -1, -23, -23, 23) ] & (1, 46) & \mbox{$T_{47} \rtimes C_2$} \\
96 &  [ (19, 19, 19, 19), ( -5, -5, 3, 3 )] & (19, 38) & \mbox{TA, primitive} \\ 
96 &  [ (19, 19, 19, 19), ( -1, -1, -1, 1) ]  & (76, 19/2) & \mbox{TA, primitive} \\
96 & [ (19, 19, 19, 19 ), ( -1, 1, 1, -1) ]  & (76, 19/2) & \mbox{Not TA, primitive}\\ 
96 &  [ (19, 19, 19, 19), ( 3, 3, 3, -5 )]  & (19, 38) & \mbox{TA, primitive} \\ 
96 & [ (31, 31, 31, 1 ), ( -1, -1, -1, 1) ]  & (31, 32) & \mbox{$C_3:K_{32}$, TA} \\
98 & [ (1, 24, 24, 24), ( -1, 24, 24, -24) ]  & (1, 48) & \mbox{$U_{49} \rtimes C_2$} \\
99 & [ (4, 4, 10, 40), ( 4, 4, -1, -4) ]  & (10, 44) & \mbox{$U_9 : K_{11}$} \\
102 & [ (1, 25, 25, 25), ( -1, -25, -25, 25) ]  & (1, 50) & \mbox{$T_{51} \rtimes C_2$} \\ 
105 & [ (34, 34, 34, 1 ), ( -1, -1, -1, 1) ]  & (34, 35) & \mbox{$C_3 : K_{35}$, TA} \\
106 & [ ( 1, 26, 26, 26), ( -1, 26, 26, -26) ]  & (1, 52) & \mbox{$U_{53} \rtimes C_2$} \\ 
110 & [ (1, 27, 27, 27 ), ( -1, -27, -27, 27) ] & (1, 54) & \mbox{$T_{55} \rtimes C_2$} \\ 
110 & [ (2, 2, 21, 42), ( 2, 2, -1, -2) ]  & (21, 44) & \mbox{$U_5:K_{22}$, TA} \\
112 & [ (15, 45, 45, 3 ), ( -1, -3, -3, 3) ]  & (15, 48) & \mbox{$T_7:K_{16}$, TA} \\ 
114 & [ (1, 28, 28, 28 ), ( -1, 28, 28, -28) ] & (1, 56) &\mbox{$U_{57} \rtimes C_2$} \\  
114 & [ (37, 37, 37, 1 ), ( -1, -1, -1, 1) ] & (37, 38) & \mbox{$C_3:K_{38}$, TA} \\  
118 & [ (1, 29, 29, 29), ( -1, -29, -29, 29) ] & (1, 58) & \mbox{$T_{59} \rtimes C_2$} \\ 
120 & [ (17, 17, 51, 17 ), ( -3, -3, 3, 1 )]  & (51, 34) & \mbox{ TA, primitive } \\ 
120 & [ (17, 17, 51, 17), ( 1, 1, 3, -3) ]  & (51, 34) & \mbox{ Not TA, primitive } \\
122 & [ (1, 30, 30, 30), ( -1, 30, 30, -30) ] & (1,60) & \mbox{$U_{61} \rtimes C_2$} \\  
123 & [ (40, 40, 40, 1 ), ( -1, -1, -1, 1) ] & ( 40, 41) &\mbox{$C_3:K_{41}$, TA} \\   
126 & [ (1, 31, 31, 31), ( -1, -31, -31, 31) ] & ( 1, 62 ) & \mbox{$T_{63} \rtimes C_2$} \\ 
130 & [ (1, 32, 32, 32 ), ( -1, 32, 32, -32) ]  &( 1, 64 ) & \mbox{$U_{65} \rtimes C_2$} \\ 
132 & [ (43, 43, 43, 1 ), ( -1, -1, -1, 1) ] & ( 43, 44 ) &\mbox{$C_3:K_{44}$, TA} \\   
134 & [ (1, 33, 33, 33 ), ( -1, -33, -33, 33) ] & ( 1, 66 ) & \mbox{$T_{67} \rtimes C_2$} \\ 
135 &  [ (2, 2, 26, 52), ( 2, 2, -1, -2 )] & ( 26, 54 ) &  \mbox{$U_5:K_{27}$, TA} \\
138 & [ (1, 34, 34, 34 ), ( -1, 34, 34, -34) ] & ( 1, 68 ) & \mbox{$U_{69} \rtimes C_2$} \\ 
141 & [ (46, 46, 46, 1), ( -1, -1, -1, 1) ] & ( 46, 47 ) & \mbox{$C_3:K_{47}$, TA} \\ 
142 & [ (1, 35, 35, 35), ( -1, -35, -35, 35) ] & ( 1, 70 ) & \mbox{$T_{71} \rtimes C_2$} \\ 
142 & [ (15, 21, 35, 35), ( -15, -21, -35, 35) ] & ( 1, 70) & \mbox{$T_{71}:C_2$, TA (see Lemma \ref{2ks}) }\\
143 &  [ (12, 60, 60, 5 ), ( -1, -5, -5, 5) ] & ( 12, 65) & \mbox{$11$-array, TA}\\
144 &  [ (39, 39, 39, 13 ), ( -9, -9, -9, 13) ] & ( 13/3, 208/3 ) &  \mbox{$T_{27}:D_{16/3}$, TA}  \\
146 & [ (1, 36, 36, 36 ), ( -1, 36, 36, -36) ] & ( 1, 72 ) &\mbox{$U_{73} \rtimes C_2$} \\  
150 & [ (49, 49, 49, 1 ), ( -1, -1, -1, 1) ] & ( 49, 50 ) & \mbox{$C_3:K_{50}$, TA} \\ 
150 & [ (1, 37, 37, 37), (-1, -37, -37, 37) ] & ( 1, 74 ) & \mbox{$T_{75} \rtimes C_2$} 
\end{array} $$}

\medskip

\begin{example}{\rm
One family of noncommutative association schemes of rank $6$ identified by Hanaki and Zieschang \cite{HZ} in which every member has a nonnormal closed subset $\{b_0,b_1\}$ of rank $2$ has linear character values $[\delta,\phi] = [(\delta_1,\delta_2,\delta_3,\delta_4),(-1,-c,-c,c)]$, with $m_{\phi}=(m_{\chi}+1)(\delta_1-1)+1$ and $c \in \mathbb{Z}$.   Integralilty of $\chi$ requires that $m_{\chi}$ divide $c m_{\phi}-\delta_i$ for $i=2,3,4$.  The association schemes in this family are semidirect products of a symmetric normal closed subset $U_{1+\delta_2+\delta_3}$ of rank $3$ with a nonnormal closed subset $K_{1+\delta_1}$ of rank $2$.   The members of one subfamily of these schemes are of  the form $U_{4k+1} \rtimes C_2$ and have order $8k+2$.   These appear in the table for every $k \ge 1$.    Alternating with these is a family of association schemes that are semidirect products of an anti-symmetric normal closed subset of rank $3$ with the thin scheme $C_2$.  These have algebraic structure $T_{4k-1} \rtimes C_2$, order $8k-2$, and occur for every $k \ge 1$. 
}\end{example} 

\begin{example}{\rm
Another family discussed in \cite{HZ} has members that are the semidirect product of a symmetric normal closed subset $U_m$ of rank $3$ with a closed subset $K_t$ of rank $2$.  We identify these as $U_m \rtimes K_t$ when no other parameter pattern is in agreement. Members of this family have  parameter sets 
$$
[\delta,\phi] = \Big[ (\delta_1,\delta_1,\delta_3,\delta_4),
\big( \delta_1,\delta_1,\frac{-\delta_3}{m_{\phi}},\frac{-\delta_4}{m_{\phi}} \big) \Big]. 
$$  

We did find some table algebras that have a normal symmetric closed subset $U$ of rank $3$ for which the RBA is a non-split extension of $U_{2u+1}$ by a table algebra of rank $2$.  Such table algebras were observed recently by Yoshikawa \cite{Yoshikawa14}.  The parameter sets are in the above form with $\delta_1$ even.  Since the quotient has rank $2$, we identify these in the table by $U_m:K_t$ or $U_m:D_t$.  The first of these has the form $U_5:K_7$.  It is a table algebra only, since it does not appear in the classification of association schemes of order $35$ and rank $6$.   In fact, according to Yoshikawa, it is not known if any of these table algebras are realized by association schemes. 
}\end{example}

\begin{example}{\rm
Another family of noncommutative rank $6$ table algebras identified by Hanaki and Zieschang in \cite{HZ} has an anti-symmetric thin closed subset $\mathbf{C}=C_3$ of rank $3$ and no other nontrivial closed subsets.  The quotient $\mathbf{B}/\!\!/C_3$ is a rank $2$ scheme $K_{n/3}$.  These have parameters $[\delta,\phi] =[(\ell,\ell,\ell,1),(-1,-1,-1,1)]$ and order $3(\ell+1)$, for all $\ell \equiv 1 \mod 3$.   These are not always association schemes.  In particular the table algebra $C_3:K_{11}$ of order $33$ is not realized by a scheme.  This was noted in \cite{HZ}.  
}\end{example}

\begin{example}{\rm 
The articles of Hanaki-Zieschang \cite{HZ} and Asaba-Hanaki \cite{AH} provide character tables of noncommutative integral table algebras of rank six, the parameters of which, for certain choices of $q$ and $r$, correspond to the projective geometry $PG_{r-1}(r,q)$.  The corresponding values $[\delta,\phi]$ for these table algebras are: 
$$ 
[\delta, \phi] = \Big[ \big(\frac{q(q^r-1)}{q-1},\frac{q^{2r+1}(q^r-1)}{q-1)},\frac{q(q^r-1)^3}{(q-1)^3)},\frac{q^{r+1}(q^r-1)^2}{(q-1)^2} \big),(-1,-q^{r-1},-q^{r-1},q^{r-1}) \Big]. 
$$
These are identified in the table with the label $PG(r,q)$.   If $r=3$ these coincide with Coxeter schemes in \cite{HZ} generated by a pair of noncommuting scheme reflections in $\{b_1,b_2,b_3\}$, which are indicated in the table with the label $CS(b_i,b_j)$.  
}\end{example}

\begin{example} {\rm
The imprimitive rank $6$ association schemes introduced by Drabkin and French \cite{DF} that arise from complete $p$-arrays for Mersenne primes $p$ have linear character values
$$
[\delta,\phi] = \Big[ \big( p+1, \frac{p^2-1}{2},\frac{p^2-1}{2},\frac{p-1}{2}\big),
 (-1,-\frac{p-1}{2},-\frac{p-1}{2},\frac{p-1}{2}) \Big]. 
$$
As one of the structure constants is $\lambda_{454}=\frac{p-3}{4}$, we find that there is an integral table algebra defined by these parameters for every positive odd integer $p \equiv 3 \mod 4$.   These have orders $p^2+2p$, an anti-symmetric normal closed subset of rank $3$ and order $p$, and no other closed subsets, so their algebraic structure is that of an extension $T_p:K_{p+2}$.  Three of these occur in the table, with orders $15$, $63$, and $143$.    
}\end{example}
 
Some of our examples are imprimitive with a normal noncentral closed subset $K_m$ of rank $2$ and order $m$ and no other nontrivial closed subsets.  The quotient is an anti-symmetric table algebra of rank $3$.  We identify these in our list as being of type $K_m: T_{n/m}$.   We also see some other instances of table algebras that are nonsplit extensions of an anti-symmetric normal closed subset by a rank $2$ quotient.  These are identified with $T_m:C_2$, $T_m:K_t$ or $T_m:D_t$ according to quotient type.   
  
\begin{example} {\rm 
Our computer search revealed an ever-increasing number of integral RBAs of rank $6$ with $\phi = (0,0,-1,0)$ (up to order).   In this case the integral RBA $(A,\mathbf{B})$ is the circle product of a normal closed subset $\mathbf{N}=\{b_0,b_3\}$ with a non-commutative rank $5$ rational RBA.  $\mathbf{N}$ is a rank $2$ RBA of order $m$, so we indicate the members of this family by $E_{n/m} \circ K_m$ in the table.  To see this, note that with these values of $\phi$ we must have $e_{\phi}=\frac{m_{\phi}}{n}(b_0-\frac{1}{\delta_3}b_3)$, and the fact that $e_{\phi}^2=e_{\phi}$ implies $\{b_0,b_3\}$ is a closed subset.  It follows that   
$$ 
b_3^2 = \delta_3 b_0 + (\delta_3-1)b_3, m_{\phi} = \frac{n\delta_3}{\delta_3+1}, \mbox{ and } \chi_3 = 2\delta_3. 
$$ 
Further $e_{\phi}$ in the center of $A$ implies $b_3$ is central in $A$.  This means our matrix realization from \S 3 has  
$$ 
b_3=\Big( \delta_3,-1,\begin{bmatrix} \delta_3 & 0 \\ 0 & \delta_3 \end{bmatrix}\Big). 
$$ 
Therefore, for $i \ne 0,3$, $b_ib_3 = b_3b_i = \delta_3b_i$. So the RBA
$(A, \b, \delta)$ is a wreath product $(\b, \bn)$; cf. \cite[Definition 1.2]{BX} and \cite{am}. 
Therefore, the cosets of $\mathbf{N}$ in $\mathbf{B}$ are $\mathbf{N}$, $\mathbf{N}b_1$, $\mathbf{N}b_2$, $\mathbf{N}b_4$, and $\mathbf{N}b_4^*$, and 
we have a quotient RBA
$$ 
\mathbf{B}/\!\!/\mathbf{N} = \{\bar{b}_0, \bar{b}_1, \bar{b}_2, \bar{b}_4, \bar{b}_4^*\}, 
$$
where $\bar{b}_0 = \frac{1}{1+\delta_3}(b_0+b_3)$, $\bar{b}_i = \frac{1}{1+\delta_3}b_i$ for $i \ne 0,3$.  The quotient RBA $\mathbf{B}/\!\!/\mathbf{N}$ has structure constants 
$$ \bar{b}_i\bar{b}_j = \frac{1}{(1+\delta_3)} \sum_{k \ne 3} \lambda_{ijk} \bar{b}_k, $$
so it is a noncommutative RBA of rank $5$, and $\mathbf{B}$ is a table algebra precisely when $\mathbf{B}/\!\!/\mathbf{N}$ is a table algebra. 
If $\mathbf{B}$ is a table algebra, then $\mathbf{B} \cong 
(\mathbf{B}/\!\!/\mathbf{N}) \wr \mathbf{N}$ by \cite[Lemma 3.1]{x11}.  
Noncommutative rank $5$ RBAs were classified in \cite{HMX}. 
}\end{example}

\section{Primitive table algebras} 

The table in Section \ref{s-ex}
contains four examples of primitive integral table algebras: one of order $81$, two of order $96$ and one of order $120$. We would like to know whether they could be Bose-Mesner algebras of association schemes. For this purpose we need the following result.    Throughout this section $(A,\mathbf{B})$ is a noncommutative rank $6$ table algebra with positive degree map $\delta$ and standard basis $\mathbf{B}$ determined by the parameters $\delta_i$, $\phi_i$, $i=1,\dots,4$, as explained in Section~\ref{s-exist}. 

\begin{proposition}\label{180716b}  The center of $A$ is a fusion subalgebra of $(A,\mathbf{B})$ iff the set $\{\phi_i/\delta_i\}_{i=0}^5$ contains at most three distinct elements.
\end{proposition}

\begin{proof} 
If the center $Z(A)$ is a fusion subalgebra, then there exists a $^*$-invariant partition of $\{1,...,5\}$ into two classes, say $I$ and $J$, such that the elements $b_0,b_I:=\sum_{i\in I} b_i,b_J:=\sum_{j\in J} b_j$ form a table basis of $Z(A)$, as the dimension of $Z(A)$ is $3$.
On the other hand, the idempotents $e_\delta,e_\phi, e_\chi = b_0 - e_\delta - e_\phi$ form another basis of $Z(A)$. Therefore $e_\phi$ is a linear combination of $b_0$, $b_I$ and $b_J$, and hence $|\{\phi_i/\delta_i,|\,i=0,...,5\}|\leq 3$.  Assume now that $|\{\phi_i/\delta_i,|\,i=0,...,5\}|\leq 3$. Pick an element $\alpha\in\{\phi_i/\delta_i,|\,i=0,...,5\}$ distinct from $1$ and define $I:=\{i\in\{1,...,5\}\,|\,\phi_i/\delta_i = \alpha\}$. It follows from $\phi_4/\delta_4 =\phi_5/\delta_5$ that the set $I$ is $^*$-invariant.  As we have already seen before, $I$ is a proper subset of $\{1,...,5\}$. Define $J:=\{1,...,5\}\setminus I$. Then $Z(A)\subseteq \spa\{ b_0,b_I,b_J\}$. Comparing the dimensions we conclude that $Z(A) = \spa\{ b_0,b_I,b_J\}$.
\end{proof}

\begin{proposition}\label{260716b} Assume that $Z(A)$ is a fusion subalgebra. Let $I\cup J=\{1,2,...,5\}$ be the corresponding partition of the index set. Then for each non-empty proper subset $K\subset I$ such that $b_K^* = b_K$,
$\spa\{ b_0, b_K, b_{I \setminus K}, b_J \}$ is a commutative real fusion subalgebra of rank $4$, and its multiplicities are $1,m_\phi,m_\chi,m_\chi$.
\end{proposition}

\begin{proof} 
Since $Z(A)=\spa\{ b_0, b_I,b_J\}$, the elements $b_K,b_{I\setminus K}, b_J$
commutes with each other under the multiplication of $A$. Thus, the subalgebra
$W:=\langle b_0,b_K,b_{I\setminus K}, b_J\rangle$ generated by 
$b_0,b_K,b_{I\setminus K}, b_J$ is a commutative subalgebra of dimension at least $4$. But any commutative subalgebra of $A$ has dimension at most four.  Therefore $W = \spa\{ b_0, b_K, b_{I \setminus K}, b_J \}$. 
It follows from the assumption that every basis element of $W$ is real.  
Thus $W$ is a commutative
real fusion subalgebra of $A$, as desired. Also $W$ is a \ta.
The restriction of $\phi$ to $W$ must remain irreducible, and the restriction
of the nonlinear irreducible character $\chi$ of $A$ to $W$ must produce two distinct irreducibles.  The  restriction of the standard feasible trace of $A$ to $W$ will be the standard feasible trace of $W$, and its irreducible constituents will occur with the multiplicities $1,m_\phi,m_\chi,m_\chi$.
\end{proof}

All parameter sets of primitive integral table algebras described in Table 1 satisfy the assumptions of Proposition \ref{180716b}. Therefore their centers are  fusion subalgebras of rank three.  If $n=81$, then the center has degrees $1,20,60$. According to A.~Brouwer's data~\cite{brouwer} there exists a unique strongly regular graph with these parameters. By Proposition \ref{260716b} there exists a real fusion subalgebra of rank $4$ with degrees $1,20,20,40$ and multiplicites $1,20,30,30$. According to the results of E.~van Dam such a scheme does not exist \cite[pg.~90]{vandam}. 
For $n=96$ there are two parameter sets. As table algebras,
their centers have the same degrees $1,38,57$. According to A. Brouwer's data~\cite{brouwer} such a strongly regular graph does not exist.  So this parameter set cannot be realized as a Bose-Mesner algebra of an association scheme.  In the last case $n=120$, the center has degrees $1,34,85$. According Brouwer's table \cite{brouwer} it is not known whether there exists a strongly regular graph with these parameters.   However, when we compute our structure constants we find that $\lambda_{414}=7$ and $\delta_4=51$.  We claim that in an association scheme the value of $\lambda_{iji}\delta_i$ should be even when $b_j$ is $*$-symmetric.  (The second author believes that this is well-known.  As we were unable to find a reference we provide an argument for this below.)  So from our noncommutative rank $6$ table algebra classification, it follows that there is no such association scheme. 

\begin{lemma}
Let $s_i$ and $s_j \ne 1_X$ be distinct relations in an association scheme $(X,S)$.  Let $\lambda_{ijk}$ be the intersection numbers of $(X,S)$.  If $s_j$ is symmetric, and the valency $\delta_i$ of $s_i$ is greater than $1$, then $\lambda_{iji}\delta_i$ must be even. 
\end{lemma} 

\begin{proof} 
Let $x \in X$.  Let $xs_i = \{y \in X : (x,y) \in s_i\}$ be the $s_i$-neighborhood of $x$.  Consider the graph $\Gamma_{i,j}$ on $xs_i$ that is induced by the relation $s_j$.  Then $\Gamma_{i,j}$ is a graph on $|xs_i|=\delta_i > 1$ points, and $(y,z) \in E(\Gamma_{i,j})$ if and only if
  $(x,y), (x,z) \in s_i$ and $(y,z) \in s_j$ . Since $s_j$ is symmetric, the graph $\Gamma_{i,j}$ is an ordinary undirected graph.  Furthermore, $\Gamma_{i,j}$ is regular of valency $\lambda_{iji}$ because for every $y \in xs_i$, 
$$|\{z \in xs_i : (y,z) \in s_j \}| = |\{z \in xs_i : (z,y) \in s_j\}| = \lambda_{iji}. $$
It is easy to see that the usual adjacency matrix of a $k$-regular graph on $n > 1$ points has $nk$ entries equal to one, and the fact that the graph is symmetric implies that this number must be even.  Applying this to $\Gamma_{i,j}$ we deduce that $\delta_i \lambda_{iji}$ must be even. 
\end{proof}  

\medskip
The authors also want to mention that A.~Munemasa, after hearing a preliminary version of these results, searched primitive permutation groups of small degrees and found that there is no primitive noncommutative Schurian scheme of rank 6 with orders less than 1600.
  
\section{Noncommutative rank $6$ TAs with $m_{\phi}=1$ } 

\bigskip
With the notation of Section 3, assume $\b = \{(\delta_i,\phi_i,B_i), i=0, \dots, 5\}$ is a standard table basis for a noncommutative rank $6$ RBA with $m_{\phi}=1$.   From our formulas (see Corollary~\ref{chartable}) we see that this extra condition $m_{\phi}=1$
implies $n = \sum_i (\phi_i^2 /\delta_i)$ and $m_{\chi} = n/2-1$.  

The purpose of this section is to classify all noncommutative rank $6$ table algebras that have $m_{\phi}=1$.  We show that they separate nicely into three families.  We show that one of these families is never realized by an association scheme, and give integrality conditions for the two of these families that can be realized by association schemes. 
 
Recall that the kernel of $\phi$ is $\ker \phi := \{ b_i \in \mathbf{B} : \phi(b_i) = \delta_i\phi(b_0) \}$. For any table algebra $(A,\mathbf{B})$ and any irreducible character $\xi$ of $A$, $\ker \xi$ is a closed subset of $\b$ (see \cite[Theorem 4.2]{x08}) and $m_\xi \ge \xi(1)$ (see \cite[Theorem 1.1]{x}).  The first condition can fail for noncommutative rank $6$ RBAs with negative structure constants.  For example,  there is a noncommutative rational rank $6$ RBA of order $6$ with linear character values $[\delta,\phi] = [(1,1,1,1),(1,-2,0,0)]$ for which $\ker \phi = \{b_0,b_1\}$ is not  a closed subset of $\b$, because 
$$ b_1^2 = b_0 - \frac12 b_1 -\frac14 b_2 + \frac14 (b_3 +b_4 +b_5). $$
The extra arithmetic properties satisfied by integral table algebras will be essential to this classification.  All the properties of integral table algebras used here  can be found in \cite{B09} or \cite{x11}.   We write $o(\bn) = \sum\limits_{b_i \in \bn} \delta_i$ for the order of a closed subset $\bn$ of $\b$. 

\begin{lemma}
\label{l-phi}
Let $(A,\b)$ be a noncommutative rank $6$ table algebra with $m_{\phi}=1$.  
Let $\mathbf{N}$ be the kernel of $\phi$.  Then the following hold.
\begin{enumerate}
\item[$(i)$] If $b_i\in\mathbf{N}, i\neq 0$, then $\phi_i = \delta_i$ and $\chi_i = -\frac{2\delta_i}{m_\chi}$.

\item[$(ii)$] If $b_i\not \in\mathbf{N}$, then $\phi_i = - \delta_i$ and $\chi_i = 0$.

\item[$(iii)$] Let $b_i b_j = \sum_{k=0}^5 \l_{ijk} b_k$.

If $\phi_i \phi_j > 0$, then $\l_{ijk} = 0$ for any $k$ with $\phi_k < 0$.

If $\phi_i \phi_j < 0$, then $\l_{ijk} = 0$ for any $k$ with $\phi_k > 0$.

\item[$(iv)$] The quotient table basis $\b/\!\!/\bn$ is an abelian group of order $2$.
\end{enumerate}
\end{lemma}

\begin{proof}
(i) follows from $\tau = \delta + \phi + m_\chi \chi$, where $\tau$ is the standard feasible trace of \ab.

(ii) It follows from
$m_\phi=1$ that $\sum_i (\phi_i^2/\delta_i) = n =\sum_i \delta_i$. But $\abs{\phi_i} \leq \delta_i$
for all $i$. Hence, we must have that $\phi_i=\pm \delta_i$, $0 \le i \le 5$. Therefore,
 $\phi_i = \delta_i$ if $b_i\in\mathbf{N}$, and $\phi_i = -\delta_i$ if $b_i\not\in\mathbf{N}$. 
Furthermore, since $0 = \tau(b_i)=\delta_i + \phi_i + m_\chi \chi_i$ for any $1 \le i
\le 5$, we see
that $\chi_i = 0$ for any $b_i\not\in\mathbf{N}$. So (ii) holds.

(iii) Applying $\delta$ and $\phi$ to the equation $b_i b_j = \sum_{k=0}^5 \l_{ijk} b_k$, we get that
\begin{equation}
\label{e-bibj-d}
\delta_i \delta_j = \sum_{k=0}^5 \l_{ijk} \delta_k \qquad \hbox{and} \qquad \phi_i \phi_j = \sum_{k=0}^5 \l_{ijk} \phi_k.
\end{equation}
If $\phi_i \phi_j > 0$, then by subtracting the second equation from the first equation in \eqref{e-bibj-d}, we see that $\sum_{\phi_k < 0} \l_{ijk} \delta_k = 0$ by (i) and (ii).   Since each $\delta_k > 0$ and $\l_{ijk} \ge 0$, we must have that $\l_{ijk} = 0$ for any $k$ with $\phi_k < 0$. But if $\phi_i \phi_j < 0$, then by adding the two equations in \eqref{e-bibj-d} together, we see that $\l_{ijk} = 0$ for any $k$ with $\phi_k > 0$. So (iii) holds. 

(iv) For any $b_i, b_j \notin \bn$, $(b_i/\!\!/\bn)(b_j/\!\!/\bn) = b_0/\!\!/\bn$ by (ii) and (iii). So (iv) holds.
\end{proof}

\begin{lemma}\label{150516a} 
Let $(A,\b)$ be a noncommutative rank $6$ table algebra with $m_{\phi}=1$, and let $\bn = \ker \phi$.  Then $o(\mathbf{N})=n/2$ and $|\mathbf{N}|=3$ or $5$.  If $|\mathbf{N}|=5$, then $\b = (\b/\!\!/\mathbf{N})\wr\mathbf{N}$. 
\end{lemma}

\begin{proof} 
Since $o(\b/\!\!/\mathbf{N})=2$ by Lemma \ref{l-phi}(iv), and $o(\b/\!\!/\mathbf{N}) = o(\b)/ o(\bn)$, it follows that $o(\mathbf{N})=n/2$. 

Now $\b/\!\!/\bn$ an abelian group of order $2$ implies that $\abs{\bn} \ne 1$. If $|\mathbf{N}|=2$, without loss of generality, we may assume that $\mathbf{N}=\{b_0,b_2\}$. Pick up two arbitrary elements $b_j,b_k\in\b\setminus\mathbf{N}$. By Lemma~\ref{l-phi}(iii), the products $b_j b_k$ and $b_k b_j$ are linear combinations of $b_0$ and  $b_2$. Now it follows from $\lambda_{jk0}=\lambda_{kj0}$ and 
$$
\lambda_{kj0} +\lambda_{kj2}\delta_2 = \delta_j\delta_k =  \lambda_{jk0} +\lambda_{jk2}\delta_2
$$
that $b_j$ and $b_k$ commute.  Thus any two elements of the set $\b\setminus\mathbf{N}$ commute. Since $\b/\!\!/\bn$ is an abelian group, $\bn$ is a (strongly) normal closed subset of $\b$. So $\abs{\bn} = 2$ yields that $\bn$ is in the center $Z(A)$ of $A$. Hence $A$ is commutative, a contradiction. 

In the following we show that $\abs{\bn} \ne 4$.  For any function $f$ on $\b$, let $f|_\bn$ be the restriction of $f$ on $\bn$.  Then $(o(\bn)/o(\b))\tau|_\bn$ is the standard feasible trace of $\bn$, and 
$$
\frac{o(\bn)}{o(\b)} \tau|_\bn = \frac{1}{2} \big(\delta|_\bn + \phi|_\bn + m_\chi \chi|_\bn \big) = \delta|_\bn + \frac{n-2}{4} \chi|_\bn.
$$
So every non-principal irreducible character of $\bn$ is a constituent of $\chi|_\bn$. Since the degree of $\chi$ is $2$, $\chi|_\bn$ is either irreducible or the sum of two (not necessarily distinct) irreducible characters of $\bn$. Thus, $\bn$ has at most three distinct irreducible characters. So $\abs{\bn} \ne 4$.

Therefore, $\abs{\bn} = 3$ or $5$. If $\abs{\bn} = 5$, then $\abs{\b} = \abs{\b/\!\!/\bn} + \abs{\bn} - 1$, and hence $\b \cong_x
(\b/\!\!/\bn) \wr \bn$ by \cite[Lemma 3.1]{x11}.
\end{proof}

In keeping with our convention from Section 3, we will assume $\phi_3/\delta_3 \ne \phi_4/\delta_4$, so $b_3$ and $b_4$ lie in distinct cosets of $\bn = \ker \phi$. 
From Lemma \ref{150516a} we see that noncommutative rank $6$ table algebras with $m_{\phi}=1$ occur in three families.   The first family where $|\bn|=5$ is the wreath product of $\b/\!\!/\bn = \{\bar{b}_0, \bar{b}_3\}$ with a noncommutative rank $5$ table algebra $\bn = \{b_0,b_1,b_2,b_4,b_5\}$.   Since noncommutative rank $5$ RBAs cannot be integral by \cite[Theorem 6]{HMX}, and $\bn$ is a closed subset of $(\b/\!\!/\bn) \wr \bn$, it is impossible for members of this family to be integral table algebras.  In particular, members of this family are never realized by association schemes. 

When $|\bn|=3$, there are two families, both of which we refer to as {\it bipartite}, since the map $\phi$ separates both the set $\b$ and the ``vertices'' into two halves.  The real bipartite family has $\bn = \{b_0,b_2,b_3\}$ up to a choice of the ordering of $b_1$ and $b_2$, and the non-real bipartite family has $\bn = \{b_0,b_4,b_5\}$.   

The next results give integrality conditions for the two families of bipartite rank $6$ table algebras.  Let $\Delta$ be a $2 \times 2$ real representation affording the character $\chi$, and write $\Delta(b_i)=B_i = \begin{bmatrix} r_i & s_i \\ s_i & u_i \end{bmatrix}$ for $i=0,\dots,3$, $B_4 = \begin{bmatrix} r_4 & s_4 \\ t_4 & u_4 \end{bmatrix}$, and $B_5=B_4^{\top}$.   As $B_1$ is a real symmetric matrix, we can adjust the basis so that $s_1=0$.  Note that the matrix entries for the $B_i$ automatically satisfy several additional identities that arise from the idempotent formulas.  Since $(1,0,0_{2 \times 2}) = e_{\delta}$, $(0,1,0_{2 \times 2}) = e_{\phi}$, and $(0,0,I_{2 \times 2})= e_{\chi}$, we have 
\begin{equation}  \label{idemp-eqs}
0_{2 \times 2} = \sum_i B_i,  0_{2 \times 2} = \sum_i \frac{\phi_i}{\delta_i}B_i, \mbox{ and } \frac{n}{m_{\chi}}I_{2 \times 2} = \sum_i \frac{\chi_i}{\delta_i} B_i. 
\end{equation} 
In addition, recall that our standard feasible trace gives the identities 
\begin{equation} \label{standfeastr}
n = \tau(b_0) = 1 + m_{\phi} + 2m_{\chi} \mbox{ and } 0 = \tau(b_i) = \delta_i + m_{\phi} \phi_i + m_{\chi}\chi_i \mbox{ for } i=1,\dots,5. 
\end{equation}
The structure constants relative to $\b$ can be determined using (\ref{scs1}) once the entries of the $B_i$ are known.
  
\begin{lemma}
\label{l-phi4-n}
Suppose $\b$ is the standard basis of a real bipartite rank $6$ table algebra, for which the kernel of $\phi$ is $\bn=\{b_0,b_2,b_3\}$. 
Then $\delta_i \ge 2$ for $2 \le i \le 4$, $\delta_1 \ge \max \{ \delta_3 / \delta_2, \delta_2 / \delta_3 \}$, and
\begin{eqnarray*}
 r_1 = - u_1 = 
\sqrt{\frac{2 \delta_1 \delta_4}{\delta_2 + \delta_3} } \ep_1, & & 
r_2 = u_2 = - \frac{\delta_2}{\delta_2 + \delta_3}, \\
s_2 = - s_3 = 
\frac{\sqrt{\delta_2 \delta_3 n}}{\sqrt{2} (\delta_2 + \delta_3)} \ep_2, 
& & r_3 = u_3 = - \frac{\delta_3}{\delta_2 + \delta_3}, \\
r_4 = - u_4 = 
- \sqrt{\frac{\delta_1 \delta_4}{2(\delta_2 + \delta_3)} }\ep_1, & & 
s_4 = - t_4 = 
\frac{\ep_3}{2} \sqrt{\frac{\delta_4 n}{\delta_2 + \delta_3}  },
\end{eqnarray*}
where $\ep_1, \ep_2, \ep_3 \in \{ 1, -1 \}$.
\end{lemma}

\begin{proof}
It follows from the definition of $\bn$ that $\phi_1 < 0, \phi_2 > 0, \phi_3 > 0$, and $\phi_4 < 0$. 
In light of the idempotent equations \eqref{idemp-eqs}, Lemma \ref{l-phi} implies that
$$
1 + r_2 + r_3 = r_1 + 2 r_4 = 0, \quad
1 + u_2 + u_3 = u_1 + 2 u_4 = 0, \quad
s_2 + s_3 = s_4 + t_4 = 0.
$$
Furthermore, since $m_\chi = (n - 2)/2$ by \eqref{standfeastr}, we have that
$$
1 + \delta_2 + \delta_3 = \delta_1 + 2 \delta_4 = n/2 \quad
\hbox{and} \quad m_\chi = \delta_2 + \delta_3.
$$
Note that $r_1 = - u_1$ and $r_4 = - u_4$ by Lemma \ref{l-phi}. So it 
follows from $n \delta_1 = \tau(b_1^2) = \delta_1^2 + \phi_1^2 + 
m_\chi tr(B_1^2)
= 2 \delta_1^2 + (\delta_2 + \delta_3) (r_1^2 + u_1^2)$ that
$$
 r_1 = - u_1 = 
\sqrt{\frac{2 \delta_1 \delta_4}{\delta_2 + \delta_3} } \ep_1,  
\quad \hbox{ where } \ep_1 = \pm 1.
$$ 
Thus, $r_1 + 2 r_4 = 0$ yields that
$$
r_4 = - u_4 = 
- \sqrt{\frac{\delta_1 \delta_4}{2(\delta_2 + \delta_3)} }\ep_1.
$$
Since $0 = \tau(b_1 b_2) = \delta_1 \delta_2 + \phi_1 \phi_2 + 
m_\chi tr(B_1 B_2) = 2 r_1 m_\chi (r_2 - u_2)$, we see that 
$r_2 = u_2$. Thus, $0 = \tau(b_2) = \delta_2 + \phi_2 + 
m_\chi tr(B_2) = 2 \delta_2 + 2(\delta_2 + \delta_3) r_2$, and hence
$$
r_2 = u_2 = - \frac{\delta_2}{\delta_2 + \delta_3}, \qquad
r_3 = u_3 = - \frac{\delta_3}{\delta_2 + \delta_3}.
$$
Now $n \delta_2 = \tau(b_2^2) 
= 2 \delta_2^2 + (\delta_2 + \delta_3) (r_2^2 + u_2^2 + 2 s_2^2)$
and $s_2 + s_3 = 0$ imply that
$$
s_2 = - s_3 = 
\frac{\sqrt{\delta_2 \delta_3 n}}{\sqrt{2} (\delta_2 + \delta_3)} \ep_2, 
\quad \hbox{ where } \ep_2 = \pm 1.
$$
Finally, since $t_4 = - s_4$, and $n \delta_4 = \tau(b_4 b_4^*) =
2 \delta_4^2 + m_\chi (r_4^2 + u_4^2 + s_4^2 + t_4^2)$, we get that
$$
s_4 = - t_4 = 
\frac{\ep_3}{2} \sqrt{\frac{\delta_4 n}{\delta_2 + \delta_3}  },
\quad \hbox{ where } \ep_3 = \pm 1.
$$

Now it is straightforward to check that
$$
b_2^2 = \delta_2 b_0 +
\frac{\delta_2(\delta_2^2 - \delta_2 + \delta_2 \delta_3 - 3 \delta_3)}
{(\delta_2 + \delta_3)^2} b_2 + 
\frac{\delta_2(\delta_2 n - 2 \delta_3)}{2(\delta_2 + \delta_3)^2} b_3,
$$
and
$$
b_3^2 = \delta_3 b_0 +
\frac{\delta_3(\delta_3 n - 2 \delta_2)}{2(\delta_2 + \delta_3)^2} b_2 + 
\frac{\delta_3(\delta_3^2 - \delta_3 + \delta_2 \delta_3 - 3 \delta_2)}
{(\delta_2 + \delta_3)^2} b_3. 
$$
Thus, 
\begin{equation}
\label{e-ine-delta2}
\delta_2^2 - \delta_2 + \delta_2 \delta_3 - 3 \delta_3 \ge 0
\qquad \hbox{and} \qquad
\delta_3^2 - \delta_3 + \delta_2 \delta_3 - 3 \delta_2 \ge 0.
\end{equation}
Adding both sides of the two inequalities in \eqref{e-ine-delta2}, we 
get that $\delta_2 + \delta_3 \ge 4$. If $\delta_2 < 2$, then 
$\delta_3 > 2$, and the first inequality in \eqref{e-ine-delta2}
yields that
$$
\delta_2^2 - \delta_2 \ge (3 - \delta_2) \delta_3 > 6 - 2 \delta_2,
$$
a contradiction. So we must have that $\delta_2 \ge 2$. Similarly, 
we also have $\delta_3 \ge 2$. Furthermore, since
$$
b_1 b_4 = 
\frac{\delta_4 (\delta_1 \delta_2 + 
\ep_1 \ep_2 \ep_3 \sqrt{\delta_1 \delta_2 \delta_3})}
{\delta_2(\delta_2 + \delta_3)} b_2 + 
\frac{\delta_4 (\delta_1 \delta_3 - 
\ep_1 \ep_2 \ep_3  \sqrt{\delta_1 \delta_2 \delta_3})}
{\delta_3(\delta_2 + \delta_3)} b_3 
$$
and
$$
b_1 b_4^* = 
\frac{\delta_4 (\delta_1 \delta_2 - 
\ep_1 \ep_2 \ep_3 \sqrt{\delta_1 \delta_2 \delta_3})}
{\delta_2(\delta_2 + \delta_3)} b_2 + 
\frac{\delta_4 (\delta_1 \delta_3 + 
\ep_1 \ep_2 \ep_3  \sqrt{\delta_1 \delta_2 \delta_3})}
{\delta_3(\delta_2 + \delta_3)} b_3, 
$$
we see that $\delta_1 \ge \max \{ \delta_3 / \delta_2, 
\delta_2 / \delta_3 \}$. It is also straightforward to check that
$$
b_4 b_4^* = \delta_4 b_0 + 
\frac{\delta_4 (\delta_2 \delta_4 - \delta_2 +
\ep_1 \ep_2 \ep_3 \sqrt{\delta_1 \delta_2 \delta_3})}
{\delta_2 (\delta_2 + \delta_3)} b_2 + 
\frac{\delta_4 (\delta_3 \delta_4 - \delta_3 -
\ep_1 \ep_2 \ep_3 \sqrt{\delta_1 \delta_2 \delta_3})}
{\delta_3 (\delta_2 + \delta_3)} b_3.
$$
So either $\delta_2 \delta_4 - \delta_2 \ge \sqrt{\delta_1 \delta_2 \delta_3}$
or $\delta_3 \delta_4 - \delta_3 \ge \sqrt{\delta_1 \delta_2 \delta_3}$.
But $\delta_1 \delta_3 \ge \delta_2$ and $\delta_1 \delta_2 \ge \delta_3$.
Hence, we always have $\delta_4 \ge 2$, and the lemma holds.
\end{proof} 

The other structure constants of the \ta\ in the above lemma are given as below:

$$
b_1^2 = \delta_1 b_0 + 
\frac{\delta_1^2 - \delta_1}{\delta_2 + \delta_3} (b_2 + b_3),
$$

$$
b_1 b_2 = \frac{\delta_2 (\delta_1 - 1)}{\delta_2 + \delta_3} b_1 +
\frac{\delta_1 \delta_2 + 
\ep_1 \ep_2 \ep_3 \sqrt{\delta_1 \delta_2 \delta_3}}
{\delta_2 + \delta_3} b_4 + 
\frac{\delta_1 \delta_2 - 
\ep_1 \ep_2 \ep_3 \sqrt{\delta_1 \delta_2 \delta_3}}
{\delta_2 + \delta_3} b_4^*,
$$

$$
b_1 b_3 = \frac{\delta_3 (\delta_1 - 1)}{\delta_2 + \delta_3} b_1 +
\frac{\delta_1 \delta_3 - 
\ep_1 \ep_2 \ep_3  \sqrt{\delta_1 \delta_2 \delta_3}}
{\delta_2 + \delta_3} b_4 +
\frac{\delta_1 \delta_3 + 
\ep_1 \ep_2 \ep_3  \sqrt{\delta_1 \delta_2 \delta_3}}
{\delta_2 + \delta_3} b_4^*, 
$$

$$
b_2 b_3 = b_3 b_2 =
\frac{\delta_3(\delta_2 n - 2 \delta_3)}{2(\delta_2 + \delta_3)^2} b_2 +
\frac{\delta_2(\delta_3 n - 2 \delta_2)}{2(\delta_2 + \delta_3)^2} b_3, 
$$

$$
b_2 b_4 = 
\frac{\delta_4 (\delta_1 \delta_2 - 
\ep_1 \ep_2 \ep_3 \sqrt{\delta_1 \delta_2 \delta_3})}
{\delta_1(\delta_2 + \delta_3)} b_1 +
\frac{\delta_2 \delta_4 - \delta_2 +
\ep_1 \ep_2 \ep_3 \sqrt{\delta_1 \delta_2 \delta_3}}
{\delta_2 + \delta_3} b_4 + 
\frac{\delta_2 \delta_4}{\delta_2 + \delta_3} b_4^*,
$$

$$
b_2 b_4^* = 
\frac{\delta_4 (\delta_1 \delta_2 + 
\ep_1 \ep_2 \ep_3 \sqrt{\delta_1 \delta_2 \delta_3})}
{\delta_1(\delta_2 + \delta_3)} b_1 +
\frac{\delta_2 \delta_4}{\delta_2 + \delta_3} b_4 +
\frac{\delta_2 \delta_4 - \delta_2 -
\ep_1 \ep_2 \ep_3 \sqrt{\delta_1 \delta_2 \delta_3}}
{\delta_2 + \delta_3} b_4^*,
$$

$$
b_3 b_4 = 
\frac{\delta_4 (\delta_1 \delta_3 + 
\ep_1 \ep_2 \ep_3  \sqrt{\delta_1 \delta_2 \delta_3})}
{\delta_1(\delta_2 + \delta_3)} b_1 +
\frac{\delta_3 \delta_4 - \delta_3 -
\ep_1 \ep_2 \ep_3 \sqrt{\delta_1 \delta_2 \delta_3}}
{\delta_2 + \delta_3} b_4 +
\frac{\delta_3 \delta_4}{\delta_2 + \delta_3} b_4^*,
$$

$$
b_3 b_4^* = 
\frac{\delta_4 (\delta_1 \delta_3 - 
\ep_1 \ep_2 \ep_3  \sqrt{\delta_1 \delta_2 \delta_3})}
{\delta_1(\delta_2 + \delta_3)} b_1 +
\frac{\delta_3 \delta_4}{\delta_2 + \delta_3} b_4 + 
\frac{\delta_3 \delta_4 - \delta_3 +
\ep_1 \ep_2 \ep_3 \sqrt{\delta_1 \delta_2 \delta_3}}
{\delta_2 + \delta_3} b_4^*,
$$

$$
b_4^* b_4 = \delta_4 b_0 + 
\frac{\delta_4 (\delta_2 \delta_4 - \delta_2 -
\ep_1 \ep_2 \ep_3 \sqrt{\delta_1 \delta_2 \delta_3})}
{\delta_2 (\delta_2 + \delta_3)} b_2 + 
\frac{\delta_4 (\delta_3 \delta_4 - \delta_3 +
\ep_1 \ep_2 \ep_3 \sqrt{\delta_1 \delta_2 \delta_3})}
{\delta_3 (\delta_2 + \delta_3)} b_3,
$$

$$
b_4^2 = b_5^2 = \frac{\delta_4^2}{\delta_2 + \delta_3}(b_2 + b_3).
$$

For a \ta\ \ab, let $\aut(\b) := \{ f \mid f : \b \to \b$ is an 
exact isomorphism\} be the automorphism group of $\b$. Let $G$ be 
a subgroup of $\aut(\b)$, and $\c G$ the group algebra of $G$ over $\c$,
which is also a \ta. Then the semi-direct product of \ab\ by \cg,
$(\c G \ltimes A, G \ltimes \b)$, is a \ta\ defined as follows: 
$\c G \ltimes A := \c G \times A$, $G \ltimes \b := G \times \b$, and the 
multiplication is defined by
$$
(g, b) \cdot (h, c) := (gh, b^h c), \quad \hbox{for any } g, h \in G, \mbox{ and } b, c \in \b,
$$
where $b^h := h(b)$ is the image of $b$ under $h$. 

We can now give a characterization of real bipartite integral table algebras of rank $6$.

\begin{thm}
\label{t-phi1b}
Let $(A,\b)$ be a real bipartite rank $6$ table algebra with $\bn = \{b_0, b_2, b_3\}$.  
Then $(A,\b)$ is integral if and only if 
\begin{equation}
\label{e-int}
\delta_1 = 1, \ \ \delta_2 = \delta_3 = \delta_4, \ \ 
2 \mid \delta_2, \ \ \hbox{and} \ \ 8 \mid (n - 2).
\end{equation}
Furthermore, if \ab\ is integral, then 
$$
\b \cong_x \aut(\bn) \ltimes \bn.
$$
\end{thm} 

\begin{proof} 
First assume that \ab\ is integral.
Since the structure constants of $b_4^2$ are integers, we see that 
$(\delta_2 + \delta_3) \mid \delta_4^2$. Hence, 
$(\delta_2 + \delta_3) \mid \delta_4 (\delta_2 + \delta_3 - 2 \delta_4)$.
That is, $(\delta_2 + \delta_3) \mid \delta_4 (\delta_1 - 1)$. Now the 
structure constants for $b_1^2$ integral implies that 
$(\delta_2 + \delta_3) \mid \delta_1 (\delta_1 - 1)$. So
$(\delta_2 + \delta_3) \mid (\delta_1 + 2 \delta_4) (\delta_1 - 1)$.
That is, $(\delta_2 + \delta_3) \mid (1+ \delta_2 + \delta_3) (\delta_1 - 1)$.
Hence, $(\delta_2 + \delta_3) \mid (\delta_1 - 1)$. But 
$(\delta_2 + \delta_3)= \delta_1 - 1 + 2 \delta_4 > \delta_1 - 1 \ge 0$.
So we must have $\delta_1 = 1$. Therefore, $b_1$ is a thin element, 
$\delta_2 + \delta_3 = 2 \delta_4$, and $2 \mid \delta_4$ (because 
$(\delta_2 + \delta_3) \mid \delta_4^2$). Furthermore, we also have that 
$b_1 b_3 = b_4$ or $b_4^*$. So $\delta_3 = \sqrt{\delta_2 \delta_3}$, 
and hence $\delta_2 = \delta_3$. Now it is clear that $8 \mid (n - 2)$,
and \eqref{e-int} holds.

On the other hand, assume that \eqref{e-int} holds.
Let $k := \delta_2 \ (= \delta_3 = \delta_4)$. By renumbering $b_4, b_5$
if necessary, without loss of generality, we may assume that 
$\ep_1 \ep_2 \ep_3 = 1$. Then 
the products of two elements in $\b$ are given as follows:
\begin{eqnarray*}
& & b_1^2 = b_0, \ b_1 b_2 = b_4, \ b_1 b_3 = b_4^*, \ b_1 b_4 = b_2, \
b_1 b_4^* = b_3, \\
& & b_2^2 = k b_0 + \frac{k - 2}{2} b_2 + \frac{k}{2} b_3, \ 
b_2 b_3 = \frac{k}{2}(b_2 + b_3), \ b_2 b_4 = \frac{k}{2}(b_4 + b_4^*), \
b_2 b_4^* = k b_1 + \frac{k}{2} b_4 + \frac{k - 2}{2} b_4^*, \\
& & b_3^2 = k b_0 + \frac{k}{2} b_2 + \frac{k - 2}{2} b_3, \ 
b_3 b_4 = k b_1 + \frac{k - 2}{2} b_4 + \frac{k}{2} b_4^*, \
b_3 b_4^* = \frac{k}{2}(b_4 + b_4^*), \\
& & b_4^2 = b_5^2 = \frac{k}{2}(b_2 + b_3), \ 
b_4 b_4^* = k b_0 + \frac{k}{2} b_2 + \frac{k - 2}{2} b_3, \ 
b_4^* b_4 =  k b_0 + \frac{k - 2}{2} b_2 + \frac{k}{2} b_3. 
\end{eqnarray*}
So \ab\ is integral, and $\aut(\bn) = \{ \ell, \sigma \}$,
where $\ell$ is the identity map, and $\sigma$ interchanges $b_2$ and 
$b_3$. Now it is straightforward to check that the map defined by
$$
(\ell, b_i) \mapsto b_i, \ i = 0, 2, 3, \ \
(\sigma, b_0) \mapsto b_1, \ (\sigma, b_2) \mapsto b_4, \
(\sigma, b_3) \mapsto b_4^*
$$
is an exact isomorphism between $\b$ and $\aut(\bn) \ltimes \bn$. 
\end{proof}

One sees several examples of the real bipartite family in Table $1$ that are listed under the heading $U_{4k+1} \rtimes C_2$.  The above theorem implies these are the only type. 

\medskip
Now we consider integrality conditions for the family of non-real bipartite rank $6$ table algebras. 

\begin{lemma} \label{l-phi4-p}
If $\mathbf{N}=\{b_0,b_4,b_5\}$, then 
$\phi_i < 0$ for $i = 1, 2, 3$, 
$\max \{ \frac{\delta_2}{\delta_3}, \frac{\delta_3}{\delta_2} \}
\le \delta_1 \le \delta_2 \delta_3$, and 
by renumbering $b_4$ and $b_5$ if necessary, we have
$$
B_4 = \begin{bmatrix} - \frac{1}{2} & \frac{\sqrt{n}}{2\sqrt{2}} \cr
   - \frac{\sqrt{n}}{2\sqrt{2}} & -\frac{1}{2} \end{bmatrix}.
$$
Furthermore, 
\begin{eqnarray*}
r_1 = - u_1 = \ep_1 \sqrt{\frac{\delta_1 (\delta_2 + \delta_3)}{2 \delta_4} }, & &
r_2 = -u_2 = - \ep_1 \frac{ \sqrt{\delta_1} \delta_2}
  { \sqrt{2 \delta_4 (\delta_2 + \delta_3)}}, \\
r_3 = -u_3 = - \ep_1 \frac{ \sqrt{\delta_1} \delta_3}
  { \sqrt{2 \delta_4 (\delta_2 + \delta_3)}}, & &
s_2 = -s_3 = \frac{\ep_2}{2} \sqrt{ \frac{\delta_2 \delta_3 n}
{\delta_4 (\delta_2 + \delta_3)}}, 
\end{eqnarray*}
where $\ep_1, \ep_2 \in \{ 1, -1 \}$. 
\end{lemma}

\begin{proof}
By Lemma \ref{l-phi} the numbers 
$\phi_1, \phi_2, \phi_3$ are all negative.
By Lemma \ref{l-phi}, the idempotent equations \eqref{idemp-eqs} imply that 
$$
r_1 + r_2 + r_3 = 0, \ u_1 + u_2 + u_3 = 0, \
r_4 = u_4 = -1/2, \ s_2 + s_3 = 0, \ s_4 + t_4 = 0.
$$
So from \eqref{standfeastr} we get that
\begin{equation}
\label{e-m-chi}
m_\chi = 2 \delta_4, \ n = 2 + 4 \delta_4, \ 
1 + 2 \delta_4 = \delta_1 + \delta_2 + \delta_3 = n/2.
\end{equation}
Also by Lemma \ref{l-phi}, we have $b_4 b_4^* = \delta_4 b_0 + \l_{444}(b_4 + b_4^*)$.
Hence, $\delta_4^2 = \delta_4 + 2\delta_4 \l_{444}$. So $\l_{444} = (\delta_4 - 1)/2$,
and
\begin{equation}
\label{e-b4b5}
b_4 b_4^* = \delta_4 b_0 + \frac{\delta_4 - 1}{2}(b_4 + b_4^*).
\end{equation}
Note that
$$
B_4 = \begin{bmatrix} - \frac{1}{2} & s_4 \cr
   -s_4 & -\frac{1}{2} \end{bmatrix} \quad \hbox{and} \quad
B_4 B_4^\top = \begin{bmatrix}  \frac{1}{4} + s_4^2 &  0 \cr
   0 & \frac{1}{4} + s_4^2 \end{bmatrix}.
$$
So by \eqref{e-b4b5}, $1/4 + s_4^2 = \delta_4 - (\delta_4 - 1)/2$. Thus, 
$s_4 = \pm \frac{\sqrt{2 \delta_4 + 1}}{2} = \pm \frac{\sqrt{n}}{2 \sqrt{2}}$. 
By renumbering $b_4$ and $b_5$ if necessary, we may assume that 
$s_4 = \frac{\sqrt{n}}{2 \sqrt{2}}$. So 
\begin{equation}
\label{e-b4}
B_4 = \begin{bmatrix} - \frac{1}{2} & \frac{\sqrt{n}}{2 \sqrt{2}} \cr
   -\frac{\sqrt{n}}{2 \sqrt{2}} & -\frac{1}{2} \end{bmatrix}.
\end{equation} 
Furthermore, $b_4^2 = \frac{\delta_4 - 1}{2}b_4 +
\frac{\delta_4 + 1}{2}b_4^*$.

 Note that $r_i + u_i = 0$ for $1 \le i \le 3$ by Lemma \ref{l-phi}.
Since $n \delta_1 = \tau(b_1^2) = \delta_1^2 + \phi_1^2 + m_\chi tr(B_1^2)
= 2 \delta_1^2 + 2 \delta_4 (r_1^2 + u_1^2)$,  
we see that $r_1^2 = u_1^2 = \frac{\delta_1 (n - 2 \delta_1)}{4 \delta_4} =
\frac{\delta_1 (\delta_2 + \delta_3)}{2 \delta_4}$ by \eqref{e-m-chi}. Hence,
$$
r_1 = - u_1 =  \ep_1 \sqrt{\frac{\delta_1 (\delta_2 + \delta_3)}{2 \delta_4} },
\qquad \hbox{where } \ep_1 = \pm 1. 
$$
Now $0 = \tau(b_1 b_2) = \delta_1 \delta_2 + \phi_1 \phi_2 + m_\chi tr(B_1 B_2)
= 2 \delta_1 \delta_2 + 2 \delta_4(r_1 r_2 + u_1 u_2)$ and $r_2 = - u_2$. So we have that
$$
r_2 = - u_2 = - \ep_1 \frac{ \sqrt{\delta_1} \delta_2}
  { \sqrt{2 \delta_4 (\delta_2 + \delta_3)}}.
$$
Also it follows from $r_1 + r_2 + r_3 = 0$ that
$$
r_3 = - u_3 = - \ep_1 \frac{ \sqrt{\delta_1} \delta_3}
  { \sqrt{2 \delta_4 (\delta_2 + \delta_3)}}.
$$
Note that $n \delta_2 = \tau(b_2^2) = \delta_2^2 + \phi_2^2 + m_\chi tr(B_2^2)
= 2 \delta_2^2 + 2 \delta_4 (r_2^2 + u_2^2 + 2s_2^2)$, and $s_2 + s_3 = 0$. So we see that 
$$
s_2 = - s_3 = \frac{\ep_2}{2} \sqrt{ \frac{\delta_2 \delta_3 n}
{\delta_4 (\delta_2 + \delta_3)}}, \qquad \hbox{where } \ep_2 = \pm 1.
$$

Now it is straightforward to check that
$$
b_i^2 = \delta_i b_0 + \frac{\delta_i^2 - \delta_i}{2 \delta_4}(b_4 + b_4^*),
\quad i = 1, 2, 3.
$$
Furthermore,
$$
b_1 b_2 = \frac{\delta_1 \delta_2 + 
  \ep_1 \ep_2 \sqrt{\delta_1 \delta_2 \delta_3}}{2 \delta_4} b_4 +
\frac{\delta_1 \delta_2 -
  \ep_1 \ep_2 \sqrt{\delta_1 \delta_2 \delta_3}}{2 \delta_4} b_4^*,
$$
$$
b_1 b_3 = \frac{\delta_1 \delta_3 - 
  \ep_1 \ep_2 \sqrt{\delta_1 \delta_2 \delta_3}}{2 \delta_4} b_4 +
\frac{\delta_1 \delta_3 +
  \ep_1 \ep_2 \sqrt{\delta_1 \delta_2 \delta_3}}{2 \delta_4} b_4^*,
$$
and
$$
b_2 b_3 = \frac{\delta_2 \delta_3 + 
  \ep_1 \ep_2 \sqrt{\delta_1 \delta_2 \delta_3}}{2 \delta_4} b_4 +
\frac{\delta_2 \delta_3 -
  \ep_1 \ep_2 \sqrt{\delta_1 \delta_2 \delta_3}}{2 \delta_4} b_4^*.
$$
Since the structure constants are non-negative, it follows that
 $\max \{ \frac{\delta_2}{\delta_3}, \frac{\delta_3}{\delta_2} \}
\le \delta_1 \le \delta_2 \delta_3$, and (ii) holds.
\end{proof}

The other structure constants of non-real bipartite rank $6$ table algebras 
are given as below:

$$
b_1 b_4 = \frac{\delta_1 - 1}{2} b_1 + 
  \frac{\delta_1 \delta_2 + \ep_1 \ep_2 \sqrt{\delta_1 \delta_2 \delta_3}}
   {2 \delta_2} b_2 +
  \frac{\delta_1 \delta_3 - \ep_1 \ep_2 \sqrt{\delta_1 \delta_2 \delta_3}}
   {2 \delta_3} b_3, 
$$
$$
b_1 b_5 = \frac{\delta_1 - 1}{2} b_1 + 
  \frac{\delta_1 \delta_2 - \ep_1 \ep_2 \sqrt{\delta_1 \delta_2 \delta_3}}
   {2 \delta_2} b_2 +
  \frac{\delta_1 \delta_3 + \ep_1 \ep_2 \sqrt{\delta_1 \delta_2 \delta_3}}
   {2 \delta_3} b_3, 
$$
$$
b_2 b_4 = \frac{\delta_1 \delta_2 -
  \ep_1 \ep_2 \sqrt{\delta_1 \delta_2 \delta_3}}{2 \delta_1} b_1 +
\frac{\delta_2 - 1}{2} b_2 + 
\frac{\delta_2 \delta_3 +
  \ep_1 \ep_2 \sqrt{\delta_1 \delta_2 \delta_3}}{2 \delta_3} b_3,
$$
$$
b_2 b_5 = \frac{\delta_1 \delta_2 +
  \ep_1 \ep_2 \sqrt{\delta_1 \delta_2 \delta_3}}{2 \delta_1} b_1 +
\frac{\delta_2 - 1}{2} b_2 + 
\frac{\delta_2 \delta_3 -
  \ep_1 \ep_2 \sqrt{\delta_1 \delta_2 \delta_3}}{2 \delta_3} b_3,
$$
$$
b_3 b_4 = \frac{\delta_1 \delta_3 + 
  \ep_1 \ep_2 \sqrt{\delta_1 \delta_2 \delta_3}}{2 \delta_1} b_1 +
\frac{\delta_2 \delta_3 -
  \ep_1 \ep_2 \sqrt{\delta_1 \delta_2 \delta_3}}{2 \delta_2} b_2 +
\frac{\delta_3 - 1}{2} b_3,
$$
$$
b_3 b_5 = \frac{\delta_1 \delta_3 - 
  \ep_1 \ep_2 \sqrt{\delta_1 \delta_2 \delta_3}}{2 \delta_1} b_1 +
\frac{\delta_2 \delta_3 +
  \ep_1 \ep_2 \sqrt{\delta_1 \delta_2 \delta_3}}{2 \delta_2} b_2 +
\frac{\delta_3 - 1}{2} b_3.
$$

We can now give our integrality condition for non-real bipartite rank $6$ table algebras.

\begin{thm}
\label{t-phi1c}
Let $(A,\b)$ be a non-real bipartite rank $6$ table algebra with $\bn = \{b_0, b_4, b_5 \}$.
Then $(A,\b)$ is integral if and only if there are odd positive integers $\alpha$, $\gamma$, $k_1$, and $k_2$ such that 
$$
\delta_1=\alpha \gamma k_1, \  \delta_2=\alpha \gamma k_2, \  
\delta_3 = \alpha^2 k_1 k_2, \ gcd(k_1,k_2)=1, \ \alpha^2 < 2 \gamma, 
$$
and 
$$
\delta_4 \quad \hbox{divides each of} \quad
\gamma k_1 k_2, \ \gamma k_1 (\alpha \gamma k_1 - 1),\ 
\gamma k_2 (\alpha \gamma k_2 - 1), \  
k_1 k_2 (\alpha^2 k_1 k_2 - 1).
$$   
\end{thm}

\begin{proof}  
Suppose $gcd(\delta_1,\delta_2)=d$.  Write $\delta_1 = d k_1$ and $\delta_2=d k_2$.  The formula for the structure constant $\lambda_{241}$ tells us that  
$$ \frac{\delta_1\delta_2 \pm \sqrt{\delta_1 \delta_2 \delta_3}}{2 \delta_1} = \frac{d^2 k_1 k_2 \pm d\sqrt{k_1 k_2 \delta_3}}{2dk_1} = \frac{d k_1 k_2 \pm \sqrt{k_1 k_2 \delta_3}}{2k_1} $$
is a nonnegative integer.  Since $d$, $k_1$, and $k_2$ have to be odd, this implies $k_1$ divides $\sqrt{k_1 k_2 \delta_3}$.  Similarly, the fact that the structure constant $\lambda_{152}$ is a nonnegative integer implies that $k_2$ divides $\sqrt{k_1 k_2 \delta_3}$.  Since $gcd(k_1,k_2)=1$, we have that $\sqrt{k_1 k_2 \delta_3} = \alpha k_1 k_2$ for some $\alpha \in \mathbb{Z}^+$, and therefore $\delta_3 = \alpha^2 k_1 k_2$.  

Since $\lambda_{243}$ is a nonnegative integer, and 
$$ \lambda_{243} = \frac{\delta_2 \delta_3 \pm \sqrt{\delta_1 \delta_2 \delta_3}}{2 \delta_3} = \frac{d\alpha^2k_1 k_2^2 \pm d \alpha k_1 k_2}{2 \alpha^2 k_1 k_2} = \frac{d \alpha k_2 \pm d}{2 \alpha}, $$ 
we have that $\alpha$ divides $d$.  If $d = \alpha \gamma$, then we have $\delta_1= \alpha \gamma k_1$, $\delta_2=\alpha \gamma k_2$, and $\delta_3 = \alpha^2 k_1 k_2$.    

Since we are in the case where $1 - \delta_1 - \delta_2 - \delta_3 + 2\delta_4 = 0$, we have that 
$$ \alpha(\gamma k_1 + \gamma k_2 + \alpha k_1 k_2) - 2 \delta_4 = 1.$$
Therefore, $gcd(\alpha,\delta_4)=1$.  Our formulas for $\lambda_{ii4}$, $i = 1, 2, 3$ imply that $\delta_4$ divides $\delta_i(\delta_i - 1)$ for $i=1,2,3$.   So with the observation that $gcd(\alpha, \delta_4)=1$ we get that $\delta_4$ divides $\gamma k_1(\alpha \gamma k_1 - 1)$, $\gamma k_2 (\alpha \gamma k_2 - 1)$, and $k_1 k_2 (\alpha^2 k_1 k_2 - 1)$. 

From our formulas for $\lambda_{124}$ and $\lambda_{125}$, we see that 
$$\lambda_{124}-\lambda_{125}= \frac{\sqrt{\delta_1\delta_2\delta_3}}{\delta_4} \in \mathbb{Z}^+, $$
so $\delta_4$ divides $\sqrt{\delta_1 \delta_2 \delta_3} = \alpha^2 \gamma k_1 k_2$.   Since $gcd(\alpha, \delta_4)=1$, this shows that $\delta_4$ divides $\gamma k_1 k_2$.   

Let $\beta \in \mathbb{Z}^+$ be such that $\beta \delta_4 = \gamma k_1 k_2$.  Then $2\delta_4 = \delta_1 + \delta_2 + \delta_3 - 1$ implies 
\begin{equation}\label{160512a}
2 \gamma k_1 k_2 = \beta (\alpha \gamma k_1 + \alpha \gamma k_2 +  \alpha^2 k_1 k_2 - 1 ). 
\end{equation}
Since $\alpha \gamma (k_1 + k_2) > 1$, the right hand side of (\ref{160512a}) is strictly larger than $\beta \alpha^2 k_1 k_2$.  Since $\beta \ge 1$, this implies $2 \gamma > \alpha^2$.   

Conversely, for every choice of positive integers $\alpha$, $\gamma$, $k_1$, and $k_2$ satisfying the given  conditions, we can produce an admissible parameter set for which $m_{\phi}=1$ and the structure constants computed using our formulas will be nonnegative integers.  By Theorem \ref{020616a} the set of matrices produced in Section 3 is an integral table basis for an integral table algebra with the desired properties. 
\end{proof} 

It is not difficult to find three odd positive integers satisfying the conditions in Theorem \ref{t-phi1c}.  
The next Lemma covers the case where $\alpha=1$. 
 
\begin{lemma}\label{2ks}
Let $k_1$, $k_2$ be odd positive integers for which $\gamma := \dfrac{k_1 k_2 + 1}{k_1 + k_2}$ is also an odd integer.  If $(k_1+k_2)$ divides $(k_i^2-1)$ for $i=1,2$, then there is an integral table algebra $(A,\mathbf{B})$  satisfying the conditions in Theorem \ref{t-phi1c} with $\delta_1 = \gamma k_1$, $\delta_2 = \gamma k_2$, $\delta_3 = k_1 k_2$, and $\delta_4 = k_1 k_2$. 
\end{lemma} 

\begin{proof} 
The choice of $\gamma$ ensures that $2 \delta_4 + 1 = \delta_1 + \delta_2 + \delta_3$ will hold for $\delta_4 = k_1 k_2$, so the conditions in
Theorem \ref{t-phi1c} are satisfied.
\end{proof} 

\begin{example} \rm
\begin{enumerate}
\item[$(i)$] The conditions in Lemma \ref{2ks} are satisfied when $\delta_1=1$, and $\delta_2 = \delta_3 = \delta_4=k$ for an odd integer $k$.   There is one such association scheme of order $2+4k$ with $m_{\phi} = 1$ for every odd integer $k$.  These are listed as $T_{2+4k} \rtimes C_2$ in Table $1$. 

\item[$(ii)$]
The conditions in Lemma \ref{2ks} are satisfied when $k_1 \equiv 1 \mod 4$ and $k_2 = k_1 +2$.    Therefore, for every $k \equiv 1 \mod 4$ there is a noncommutative rank $6$ integral table algebra with $m_{\phi}=1$ and order $4k^2+8k+2$ that has degrees 
$$
\delta_1 = \frac12(k^2+k), \delta_2 = \frac12(k^2+3k+2), \delta_3 = \delta_4 = k^2+2k. 
$$
The smallest of these appears in Table $1$ with $k_1=5$ and $k_2=7$; it has order $142$.  

\item[$(iii)$]
The conditions in Lemma \ref{2ks} are also satisfied when
$k_1 = \frac{(s+1)^2}{2} - 1$, and $k_2 =  \frac{(s-1)^2}{2}
 - 1$, where $s$ is an odd integer greater than $3$. For $s = 5, 7$, or $9$, the paire
$(k_1, k_2)$ are $(17, 7), (31, 17)$, and $(31, 49)$, respectively.

\item[$(iv)$]
Additional pairs of coprime odd positive integers $(k_1,k_2)$ both less than $50$ that satisfy the conditions of Lemma \ref{2ks} but do not appear in (ii) or (iii) are: $(9,31),(11,19),(11,49),(13,29)$, $(15,41),(19,41)$, $(23,43)$, and $(29,41)$.  We have yet to determine if any of these are realized by association schemes.

\end{enumerate}
\end{example}

\section{Acknowledgements}
The authors would like to express their gratitude to the anonymous referee for his/her valuable  comments that significantly improved the final presentation of this article.

{\footnotesize

\medskip
{\sc A.~Herman, Department of Mathematics and Statistics, University of Regina, Regina, SK S4A 0A2, CANADA. Email:} {\tt Allen.Herman@uregina.ca}

\medskip
{\sc M.~Muzychuk, Department of Computer Science and Mathematics, Netanya Academic College, University St. 1, 42365, Netanya, ISRAEL.  Email:} {\tt muzy@netanya.ac.il}

\medskip
{\sc B.~Xu, Department of Mathematics and Statistics, Eastern Kentucky University, 521 Lancaster Ave., Richmond, KY 40475-3133, USA. Email:} {\tt bangteng.xu@eku.edu} 
}

\begin{thebibliography}{10}

\bibitem{AF} 
Z.~Arad and E.~Fisman, On table algebras, $C$-algebras, and applications to finite group theory, {\it Comm. Algebra}, {\bf 19} (1991), 2955-3009. 

\bibitem{AFM} 
Z.~Arad, E.~Fisman, and M.~Muzychuk, Generalized table algebras, {\it Israel J. Math.}, {\bf 114} (1999), 29-60.

\bibitem{AH}
Y.~Asaba and A.~Hanaki, A construction of integral standard generalized table algebras from parameters of projective geometries, {\it Israel J. Math.}, {\bf 194}, (2013), 395-408.

\bibitem{am}
Z. Arad and M. Muzychuk, Introduction, in: Standard integral table algebras generated by a non-real element of small degree, in: {\it Lecture Notes in Math.}, 1773, Springer, Berlin, 2002, pp. 1-11.
\bibitem{bhs}
S.~Bang, M.~Hirasaka, and S.-Y.~Song, Semidirect products of association schemes, {\it J. Algebraic Combin.}, {\bf 21} (2005), 23-38.

\bibitem{BI84}
E.~Bannai and T.~Ito, {\it Algebraic Combinatorics I: Association Schemes}, Benjamin/Cummings Publ., Menlo Park, 1984. 

\bibitem{B95} H. Blau, Quotient structures in $C$-algebras, {\it J. Algebra}, {\bf 175} (1995), 24-64; Erratum: {\bf 177} (1995), 297-337. 

\bibitem{B09} H. Blau, Table algebras, {\it European J. Combin.}, {\bf 30}, (2009), 1426-1455. 

\bibitem{B10} H. Blau, Association schemes, fusion rings, $C$-algebras, and reality-based algebras where all nontrivial multiplicities are equal, {\it J. Algebraic Combin.}, {\bf 31} (2010), 491-499. 

\bibitem{BC} H.~Blau and G.~Chen, Reality-based algebras, generalized Camina-Frobenius pairs, and the non-existence of degree maps, {\it Comm. Algebra}, {\bf 40} (4), (2012), 1547-1562.

\bibitem{BX} H.~Blau and B.~Xu, Irreducible characters of wreath products in reality-based algebras and applications to association schemes, {\it J. Algebra}, {\bf 412} (2014), 155-172.

\bibitem{brouwer} A. Brouwer,  Parameters of Strongly Regular Graphs, \\ {\tt https://www.win.tue.nl/\~{ }aeb/graphs/srg/srgtab.html}

\bibitem{vandam} E. R. van Dam, Three class assoication schemes, {\it J. of Algebraic Combin.}, {\bf 10} (1999), pp. 69-107.

\bibitem{DF}
B.~Drabkin and C.~French, On a class of noncommutative imprimitive association schemes of rank 6, {\it Comm. Algebra}, {\bf 43} (9), (2015), 4008-4041. 

\bibitem{FZ}
C.~French and P.-H.~Zieschang, On the normal structure of noncommutative association schemes of rank 6, {\it Comm. Algebra}, {\bf 44} (3), (2016), 1143-1170.

\bibitem{HZ} 
A.~Hanaki and P.-H.~Zieschang, on imprimitive noncommutative association schemes of order 6, {\it Comm. Algebra}, {\bf 42} (3), (2014), 1151-1199. 

\bibitem{h}
A. Hanaki,Character products of association schemes, {\it J. Algebra}, {\bf 283} (2005), 596-603.

\bibitem{HMX} A. Herman, M. Muzychuk, and B. Xu, The recognition problem for table algebras and reality-based algebras, {\it J. Algebra}, {\bf 479}, (2017), 173-191. 

\bibitem{Hig87} D.G. Higman, Coherent algebras, {\it Linear Algebra Appl.}, {\bf 93} (1987), 209-239.

\bibitem{Takesaki} M.~Takesaki, {\it Theory of Operator Algebras I}, Encyclopaedia of Mathematical Sciences, Vol. 124, Springer-Verlag, 1979.

\bibitem{Yoshikawa14} M.~Yoshikawa, On noncommutative integral standard table algebras in dimension $6$, {\it Comm. Algebra}, {\bf 42} (2014), 2046-2060.

\bibitem{x08}
B. Xu, Characters of table algebras and applications to association schemes,
{\it J. Combin. Theory Ser. A}, {\bf 115} (2008), 1358-1373.

\bibitem{x11}
B. Xu, On wreath products of C-algebras, table algebras, and association schemes,
{\it J. Algebra} {\bf 344} (2011), 268-283.

\bibitem{x}
B. Xu, Table algebras with exactly one irreducible character whose degree and multiplicity are not equal, {\it Comm. Algebra} (2017), dx.doi.org/10.1080/00927872.2017.1298775.

\bibitem{Z05} P.-H. Zieschang, {\it Theory of Association Schemes}, Springer-Verlag, Berlin, 2005. 
\end{thebibliography}
\end{document}